\documentclass[journal,twoside]{IEEEtran}
\usepackage{cite}
\usepackage{amsmath,amssymb,amsfonts}
\usepackage{algorithmic}
\usepackage{graphicx}
\usepackage{textcomp}
\usepackage{tikz}
\usetikzlibrary{decorations.markings,decorations.pathmorphing,calc,intersections}

\begin{document}

\newcommand{\inner}[2]{\left\langle #1, #2 \right\rangle}
\newcommand{\g}[1]{\boldsymbol{#1}}
\newcommand{\gd}[1]{\mathbf{#1}}
\newcommand{\drond}[2]{\frac{\partial #1}{\partial #2}}
\newcommand{\eps}[0]{\varepsilon}
\newcommand{\R}[0]{\mathbb{R}} 
\newcommand{\E}[0]{\mathbb{E}} 
\newcommand{\N}[0]{\mathcal{N}} 
\newcommand{\1}[0]{\mathbf{1}} 
\newcommand{\I}[1]{\boldsymbol{1}_{#1}}
\renewcommand{\H}[0]{\mathcal{H}} 
\newcommand{\X}[0]{\mathcal{X}} 
\newcommand{\Y}[0]{\mathcal{Y}} 
\newcommand{\U}[0]{\mathcal{U}} 
\newcommand{\V}[0]{\mathcal{V}} 
\newcommand{\C}[0]{\mathcal{C}}
\newcommand{\F}[0]{\mathcal{F}} 
\newcommand{\G}[0]{\mathcal{G}} 
\renewcommand{\L}[0]{\mathcal{L}} 
\newcommand{\Z}[0]{\mathcal{Z}} 
\newcommand{\rkhs}[0]{\left(\H, \inner{\cdot}{\cdot}_{\H}\right)}
\renewcommand{\S}[0]{\mathcal{S}} 
\newcommand{\D}[0]{\mathcal{D}} 

\newcommand{\sign}[0]{\mbox{sign}} 
\newtheorem{theorem}{Theorem}
\newtheorem{corollary}{Corollary}
\newtheorem{lemma}{Lemma}
\newtheorem{problem}{Problem}
\newtheorem{definition}{Definition}
\newtheorem{proposition}{Proposition}
\newtheorem{example}{Example}
\newtheorem{remark}{Remark}
\newcommand{\argmax}{\operatornamewithlimits{argmax}}
\newcommand{\argmin}{\operatornamewithlimits{argmin}}

\newenvironment{proof}{\paragraph*{Proof}}{\par\hfill$\square$\par}

\title{Margin theory for the scenario-based approach to robust optimization in high dimension}

\author{Fabien Lauer
\thanks{F. Lauer is with the Universit\'e de Lorraine, LORIA, CNRS, Nancy, France (e-mail: fabien.lauer@loria.fr). }
}

\maketitle

\begin{abstract}
This paper deals with the scenario approach to robust optimization. This relies on a random sampling of the possibly infinite number of constraints induced by uncertainties in the parameters of an optimization problem. Solving the resulting random program yields a solution for which the quality is measured in terms of the probability of violating the constraints for a random value of the uncertainties, typically unseen before. Another central issue is the determination of the sample complexity, i.e., the number of random constraints (or scenarios) that one must consider in order to guarantee a certain level of reliability. In this paper, we introduce the notion of margin to improve upon standard results in this field. In particular, using tools from statistical learning theory, we show that the sample complexity of a class of random programs does not explicitly depend on the number of variables. In addition, within the considered class, that includes polynomial constraints among others, this result holds for both convex and nonconvex instances with the same level of guarantees. We also derive a posteriori bounds on the probability of violation and sketch a regularization approach that could be used to improve the reliability of computed solutions on the basis of these bounds.
\end{abstract}


\maketitle

\section{Introduction}
\label{sec:introduction}

\IEEEPARstart{R}{obust} optimization is concerned with parametrized optimization problems with uncertainties on the parameters. The worst-case approach to these problems aims at computing a solution that satisfies the constraints for all possible values of the parameters within the uncertainty set. However, this often leads to an infinite number of constraints that cannot always be handled in practice. Instead, we focus on the scenario approach, in which the uncertainties are sampled to generate a finite number of random constraints. Then, the issue is to determine the probability with which the corresponding solution satisfies the constraints, or, conversely, how many samples should be drawn to guarantee a certain upper bound on the probability of violation. 

The scenario approach to robust optimization has a long history \cite{Tempo13,Campi08,Alamo09,Calafiore10,Esfahani14,Campi18b,Campi18,Garatti22}, also tightly connected with the field of robust control \cite{Calafiore06}.
This paper derives new results inspired by statistical learning theory, thus following the path initiated by \cite{Alamo09}. However, while \cite{Alamo09} focused solely on the VC-dimension and early tools from learning theory, we introduce a refined analysis based on more recent developments in this field. More precisely, the approach of \cite{Alamo09} focuses on whether the scenario constraints are satisfied or not. Instead, we introduce the concept of margin to measure the amount by which a scenario constraint is satisfied and then analyze the random programs through a capacity measure known as the Rademacher complexity \cite{Koltchinskii02,Bartlett02}, which takes into account the magnitude of the constraint function in addition to its sign. 

\subsection{Contributions}

The main contributions of the paper can be summarized as follows.

\subsubsection{Dimension-free a priori bounds} 
We derive bounds on the probability of violation that do not explicitly depend on the dimension $d$ of the optimization problem (the number of variables), and thus show that the sample complexity of a random problem is not directly related to its dimension. This stands in contrast to most of the literature. For instance, for convex problems, the sample complexity derived in \cite{Campi08,Calafiore10} grows linearly with the dimension. Extensions of these works to nonconvex problems suffer from the same issue. For instance, the approach of \cite{Esfahani14} relies on decomposing the nonconvex domain as a union of convex ones, and leads to bounds also linear in $d$ for integer programming. Other works \cite{Campi18,Campi18b,Garatti22} provide bounds on the probability of violation that do not involve $d$, and instead characterize the complexity in terms of the a posteriori length of the support subsample. However, these bounds only apply after the solution has been computed and cannot be used to estimate the sample complexity. In contrast, we provide a priori bounds on the probability of violation and sample complexity estimates. 

\subsubsection{Convex problems}
Random convex programs were thoroughly studied in \cite{Campi08,Calafiore10}. In particular, a priori bounds on the probability of violation in $O(d/N)$ for problems with $d$ variables and $N$ scenarios, were derived. It was also shown that these could not be improved, since there are convex problems for which the probability of violation precisely equals the bound. However, by restricting the class of problems to those with constraints with a particular structure, it was shown in \cite{Zhang15} that tighter bounds could be obtained. In this paper, we follow the latter while enlarging its scope: our proposed method can handle more general constraints than \cite{Zhang15} and leads to tighter bounds than \cite{Campi08,Calafiore10} in the high-dimensional regime. These hold, for instance, for all conjunctions of polynomial constraints for which the uncertain parameters are not involved in the exponents. 

\subsubsection{Nonconvex problems}
The bounds derived in this paper hold independently of the convexity of the problem. For instance, a convex problem over $d$ real variables can be reformulated over integers without impacting the statistical guarantees. More generally, our framework includes all combinations of conjunctions and disjunctions of polynomial constraints composed with Lipschitz functions without uncertainty in the exponents, over either real or integer variables.  
In addition, the performance guarantees derived in this paper hold uniformly over the entire domain defined by the constraints without uncertainties. This means that they apply without assumptions on our ability to solve the random scenario program. At the opposite, all approaches \cite{Esfahani14,Campi18} that extend the results of \cite{Calafiore10} to the nonconvex case typically apply only to the global optimizer that satisfies all scenario constraints. 

\subsubsection{A posteriori bounds and regularization}
When scenarios cannot be easily generated, a posteriori bounds can be used to assess the quality of a computed solution, as proposed in \cite{Campi18b,Garatti22} for convex problems and \cite{Campi18} in the nonconvex case. However, these work rely on the estimation of the support subsample via a rather costly procedure that involves solving multiple optimization problems, which might be computationally demanding in the nonconvex case. Here, we derive a posteriori bounds that can be computed in a straightforward manner for the same cost as a priori bounds. Inspired by these bounds, we describe a regularization procedure that could be used to search for solutions with increased generalization performance.

\subsection{Notations}
Sequences are written in bold letters with a subscript indicating their length, e.g., $\g t_N=(t_i)_{1\leq i\leq N}$. For two vectors $u$, $v\in\R^n$, the dot product is written as $u^\top v$ and $\|u\|$ denotes the Euclidean norm. However, all results below hold similarly with $\R^{n}$ replaced by a Hilbert space $\H$ equipped with the inner product $\inner{\cdot}{\cdot}_{\H}$ that induces the norm $\|u\|_{\H} =\sqrt{\inner{u}{u}_{\H}}$. The cardinality of a set $\X$ is denoted by $|\X|$ and, given two sets $\X$ and $\Y$, $\Y^{\X}$ is the set of functions of $\X$ to $\Y$. For two functions $f\in\R^{\X}$ and $\varphi\in\R^{\R}$, $\varphi\circ f$ denotes their composition, i.e., $\varphi\circ f(x) = \varphi(f(x))$, and for a class of functions $\F\subset\R^{\X}$, $\varphi\circ\F = \{\varphi\circ f,\ f\in\F\}$. Given two sets of functions, $\U$, $\V\subset\R^{\X}$, we also introduce natural notations for the sets induced by pointwise binary operations: $\U+\V = \{u+v,\ u\in\U,\ v\in\V\}$, $\U-\V = \{u-v,\ u\in\U,\ v\in\V\}$, $\max(\U,\V) = \{g\in\R^{\X} : g(x) = \max\{u(x), v(x)\},\ u\in\U,\ v\in\V\}$, $\min(\U,\V) = \{g\in\R^{\X} : g(x) = \min\{u(x), v(x)\},\ u\in\U,\ v\in\V\}$. The operator $\E_\theta$ denotes the expectation wrt. the random variable $\theta$, and we write merely $\E$ when the random variable is obvious from the context. The indicator function is written as $\I{E}$ and is $1$ when the expression $E$ is true and $0$ otherwise.

\subsection{Paper organization}

We start in Section~\ref{sec:main} by giving the precise formulation of the class of problems we consider. Then, after recalling the standard scenario approach in Sect.~\ref{sec:standard}, we introduce the margin and the proposed approach in Sect.~\ref{sec:margin}. Section~\ref{sec:analysis} contains the theoretical results, including the main generalization bounds in Sect.~\ref{sec:generalization}. The sample complexity is discussed in Sect.~\ref{sec:samplecomplexity} and the maximization of the margin in Sect.~\ref{sec:optimmargin}. A posteriori bounds are detailed in Sect.~\ref{sec:posteriori} and bounds with fast rates of convergence are discussed in Sect.~\ref{sec:fastrates}. Finally, optimization problems are considered in Sect.~\ref{sec:optim} and conclusions are given in Sect.~\ref{sec:conclusions}.

\section{Margin-based scenario approach}
\label{sec:main}

Given the set of admissible parameter values $\Theta$ accounting for the uncertainties, the prototype robust optimization problem that we consider is as follows. 

\begin{problem}[Robust program]\label{pb:worstcase}
Given a parameter set $\Theta$, a domain $\X$, a cost function $J:\X\to\R$ and a constraint function $f:\X\times\Theta\to\R$, solve
\begin{align*}
	&\min_{x\in\X} J(x)\\
	\text{s.t. } &\ f(x, \theta) \leq 0,\quad \forall \theta\in\Theta.
\end{align*}
\end{problem} 
\smallskip
Here, the domain $\X$ entails all the constraints that are not impacted by uncertainties. 
Note that no assumption is made regarding its convexity, and $\X$ can typically be a subset of $\R^d$ for continuous optimization, $\mathbb{Z}^d$ for integer programming or of $\R^{d_1}\times\mathbb{Z}^{d_2}$ for mixed-integer programs.

While the proposed approach described in this paper is generally applicable in principle to any constraint function $f$, we will also provide dedicated results for problems that can be formulated with 
\begin{equation}\label{eq:fmax}
	f(x,\theta) = \max_{k\in\{1,\dots,C\}} f_k(x,\theta)
\end{equation}
for a collection of $C$ functions 
\begin{equation}\label{eq:fk}
	f_k(x,\theta) = \psi_k(\theta)^\top \phi_k(x) + \eta_k(\theta),\quad k=1,\dots,C,
\end{equation}
based on functions $\psi_k:\Theta\to\R^{n_k}$, $\phi_k:\X\to\R^{n_k}$ and $\eta_k:\Theta\to\R$. Note that no assumption will be made regarding the convexity of these functions and of the resulting optimization problems. In fact, the constraint $f(x,\theta)\leq 0$ with the form~\eqref{eq:fmax}, which implements a conjunction of $C$ constraints (all $C$ functions $f_k(x,\theta)$ must be negative), can be replaced by the nonconvex form
\begin{equation}\label{eq:fmin}
	f(x,\theta) = \min_{k\in\{1,\dots,C\}} f_k(x,\theta),
\end{equation}
which implements a disjunction of $C$ alternatives (at least one $f_k(x,\theta)$ must be negative). More generally, we will focus on constraint functions that take the form
\begin{align}
	f(x,\theta) =\ & \rho_C\circ g_C (\ \dots \rho_3\circ  g_3(\ \rho_2\circ g_2( \ \varphi_1\circ f_1(x,\theta)\ ,\nonumber\\
		& \varphi_2\circ f_2(x,\theta)\ ) ,\ \varphi_3\circ f_3(x,\theta)) \dots\ , \varphi_C\circ f_C(x,\theta)) \label{eq:fminmax}
\end{align}
for $C-1$ binary operators
\begin{align*}
	g_k \in \{ & (a,b)\mapsto\max(a,b), (a,b)\mapsto\min(a,b),\\
		&\ (a,b)\mapsto a+b, (a,b)\mapsto a-b\}, \quad k=2,\dots, C,
\end{align*}
and univariate Lipschitz continuous functions $\rho_k:\R\to\R$, $k=2,\dots,C$, and $\varphi_k:\R\to\R$, $k=1,\dots,C$. 
Alternatively, a recursive formulation of~\eqref{eq:fminmax} is $f=f^C$, $f^1=\varphi_1\circ f_1$, and, for $k\geq 2$, $f^{k} = \rho_{k}\circ g_{k}( f^{k-1}, \varphi_k\circ f_k)$. 

This formulation allows for a rather large class of constraint functions. For instance,~\eqref{eq:fmax} is recovered for $g_k=\max$, $k=2,\dots,C$, and the boolean expression 
$$
	(f_1(x,\theta) \leq 0\ \vee\ f_2(x,\theta)\leq 0)\ \wedge\ f_3(x,\theta)\leq 0
$$
is equivalent to $\max\{ \min\{f_1(x,\theta), f_2(x,\theta)\}, f_3(x,\theta)\}\leq 0$. As another example, the function $f(x,\theta)=\sqrt{1.1+\sin(\theta x_1 x_2)} -  \cos(\theta x_2^3)$ can also be written as~\eqref{eq:fminmax}. 
Polynomial constraints could be implemented as well as sums of $C$ monomials $f_k(x,\theta)$. However, it will be more efficient to consider instead $C=1$ and let $f_1(x,\theta)$ be a polynomial with $\phi_1(x)$ gathering all its monomials and $\psi_1(\theta)$ its coefficients.\footnote{While the dimension $n_1$ of $\phi_1(x)$ could quickly grow with the dimension $d$ of $x$ or the degree of the polynomial, we should only keep the monomials with uncertainty in their coefficients as components of $\phi_1(x)$ and then append another single component to compute the sum of all the remaining terms with the corresponding entry in $\psi_1(\theta)$ set to $1$.} 

We note that, while similar in some aspects, our framework is more general than the structure imposed on the constraints by \cite{Zhang15}, which is basically limited to the form~\eqref{eq:fmax}--\eqref{eq:fk} with the maximum operator and constant dimensions $n_1=\dots=n_C$, a {\em convex} set $\X$ and {\em convex} functions $\phi_k$.

\subsection{Standard scenario approach}
\label{sec:standard}

The standard scenario approach aims at approximating the solution to Problem~\ref{pb:worstcase} by solving the following  version with a finite number of sampled constraints. 
\begin{problem}[Scenario program]\label{pb:rand}
Given a probability distribution of $\theta\in\Theta$, draw a random sample of $N$ independent parameter values $(\theta_i)_{1\leq i\leq N} \subset\Theta$ and solve
\begin{align*}
\hat{x} = &\argmin_{x\in\X} J(x)\\
\text{s.t. } & f(x,\theta_i) \leq 0 ,\quad i=1,\dots, N.
\end{align*}
\end{problem}
\smallskip

The ability of Problem~\ref{pb:rand} to provide solutions that satisfy the infinite number of constraints of Problem~\ref{pb:worstcase} is assessed via the risk, i.e., the probability of violation
\begin{equation}\label{eq:risk}
	V(x) = P\{ f(x,\theta) > 0 \},
\end{equation}
where $P$ stands for the probability distribution of $\theta$. 
We also define the empirical risk as the fraction of violated constraints among the scenarios, i.e., 
\begin{equation}\label{eq:empiricalrisk}
	\hat{V}(x) = \frac{1}{N}\sum_{i=1}^N \I{ f(x,\theta_i) > 0 } .
\end{equation}

A classical result \cite{Campi08,Calafiore10} guarantees for instance that for {\em convex} problems and $\hat{x}$ the solution to Problem~\ref{pb:rand}, 
\begin{equation}\label{eq:calafiore1}
	P^N\left\{ V(\hat{x}) \leq \epsilon \right\} \geq 1 - \delta,
\end{equation}
where
\begin{equation}\label{eq:calafiore2}
	\delta = \sum_{j=0}^{d-1} \binom{N}{j}\epsilon^j (1-\epsilon)^{N-j} .
\end{equation}
Thus, with high probability on the random draw of the scenarios $(\theta_i)_{1\leq i\leq N}$, the risk can be guaranteed to be as small as $\epsilon$. Such results can be used to compute the sample complexity, i.e., the minimal number of scenarios $N$ required to guarantee a priori a certain performance for given $\epsilon$ and $\delta$.

\subsection{Introducing the margin}
\label{sec:margin}

We first focus on feasibility problems, i.e., those for which no $J$ is to be minimized, and the guarantees regarding the probability of violation. 

In this paper, we introduce the notion of {\em margin} and the following hard-margin version of Problem~\ref{pb:rand}.
\begin{problem}[Hard-margin scenario program]\label{pb:hardmargin}
Given a probability distribution of $\theta\in\Theta$ and a margin parameter $\gamma>0$, draw a sample of $N$ independent parameter values $(\theta_i)_{1\leq i\leq N} \subset\Theta$ and 
$$
\text{Find } x\in\X,\quad 
\text{s.t. }  f(x,\theta_i) \leq -\gamma ,\quad i=1,\dots, N.
$$
\end{problem}
\smallskip
In Problem~\ref{pb:hardmargin}, the constraints are enforced with an additional margin that will ease the satisfaction of the original constraint $f(x,\theta) \leq 0$ when generalizing to other values of $\theta$, as illustrated by Fig.~\ref{fig:margin}.

When Problem~\ref{pb:hardmargin} (or more generally Problem~\ref{pb:worstcase}) is infeasible, slack variables can be introduced to allow for small violations of the margin constraints, leading to a soft-margin version of the approach.
\begin{problem}[Soft-margin scenario program]\label{pb:softmargin}
Given a probability distribution of $\theta\in\Theta$ and a margin parameter $\gamma>0$, draw a sample of $N$ independent parameter values $(\theta_i)_{1\leq i\leq N} \subset\Theta$ and solve
\begin{align*}
&\min_{x\in\X, \xi\in\R^N}  \sum_{i=1}^N \xi_i\\
\text{s.t. } & f(x,\theta_i) \leq -\gamma + \xi_i ,\quad i=1,\dots, N\\
& \xi_i \geq 0,\quad i=1,\dots,N.
\end{align*}
\end{problem}
\smallskip
As illustrated by Fig.~\ref{fig:softmargin}, the soft margin can be larger than the hard margin, which will translate into better performance guarantees in the analysis below. 

\begin{figure}
\centering
\begin{tikzpicture}
\draw[blue] (0,0) circle (1) ;
\draw[red,very thick] (-30:1) arc (-30:-100:1) ;
\fill circle (1pt) node[right] {$x_0$} ;
\draw (30:1) -- ++(-60:1.5) ; 
\draw (30:1) -- ++(120:1.5) ; 
\draw (120:1) -- ++(30:1.5) ; 
\draw (120:1) -- ++(210:1.5) ; 
\draw (170:1) -- ++(80:1.5) ; 
\draw (170:1) -- ++(260:1.5) ; 
\draw[name path=l1] (-30:1) -- ++(-120:1.5) ; 
\draw (-30:1) -- ++(60:1.5) ; 
\draw (-100:1) -- ++(-190:1.5) ; 
\draw[name path=l2] (-100:1) -- ++(-10:1.5) ; 
\fill[red,name intersections={of=l1 and l2}]
	(intersection-1) circle (2pt) node[below] {$\hat{x}$};
\end{tikzpicture}$\qquad$
\begin{tikzpicture}
\draw[blue] (0,0) circle (1) ;
\fill circle (1pt) node[right] {$x_0$} ;
\draw (30:1) -- ++(-60:1.5) ; 
\draw (30:1) -- ++(120:1.5) ; 
\draw (120:1) -- ++(30:1.5) ; 
\draw (120:1) -- ++(210:1.5) ; 
\draw (170:1) -- ++(80:1.5) ; 
\draw (170:1) -- ++(260:1.5) ; 
\draw (-30:1) -- ++(-120:1.5) ; 
\draw (-30:1) -- ++(60:1.5) ; 
\draw (-100:1) -- ++(-190:1.5) ; 
\draw (-100:1) -- ++(-10:1.5) ; 
\draw[<->] ($ (120:0.7) + (210:1) $) -- ($ (120:1) + (210:1) $) node[above] {$\gamma$}; 
\draw[dashed] (30:0.7) -- ++(-60:1.5) ; 
\draw[dashed] (30:0.7) -- ++(120:1.5) ; 
\draw[dashed] (120:0.7) -- ++(30:1.5) ; 
\draw[dashed] (120:0.7) -- ++(210:1.5) ; 
\draw[dashed] (170:0.7) -- ++(80:1.5) ; 
\draw[dashed] (170:0.7) -- ++(260:1.5) ; 
\draw[dashed,name path=l1] (-30:0.7) -- ++(-120:1.5) ; 
\draw[dashed] (-30:0.7) -- ++(60:1.5) ; 
\draw[dashed] (-100:0.7) -- ++(-190:1.5) ; 
\draw[dashed,name path=l2] (-100:0.7) -- ++(-10:1.5) ; 
\fill[blue,name intersections={of=l1 and l2}]
	(intersection-1) circle (2pt) node[above left] {$\hat{x}$};
\end{tikzpicture}
\caption{Illustration of the benefit of the margin on a toy example: the constraint $\|x-x_0\|\leq 1$ (blue circle) is implemented by an infinite number of linear constraints, $f(x,\theta)=\theta^\top (x-x_0) - 1 \leq 0$, for $\theta$ uniformly distributed in the unit circle. {\em Left:} The standard scenario approach yields a solution \textcolor{red}{$\hat{x}$} that satisfies the 5 scenario constraints (the 5 lines) but has a probability of violation of $V(\hat{x})=20\%$ that corresponds to the probability of drawing a tangent line to the circle at a point along its red part. {\em Right:} All feasible solutions to the hard-margin scenario Problem~\ref{pb:hardmargin}, such as \textcolor{blue}{$\hat{x}$}, satisfy the 5 scenario constraints with a margin $\gamma$ (the 5 dashed lines) and also satisfy the worst-case constraint (blue circle) with $V(\hat{x})=0$. \label{fig:margin}}
\end{figure}
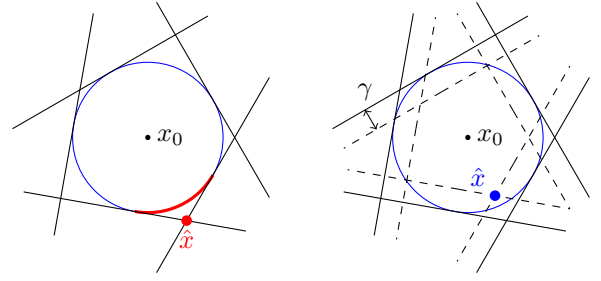
\begin{figure}
\centering
\tikzstyle{tangent}=[
decoration={
markings,
mark=
at position #1
with
	{
	\coordinate (tangent point-\pgfkeysvalueof{/pgf/decoration/mark info/sequence number}) at (0pt,0pt);
	\coordinate (tangent unit vector-\pgfkeysvalueof{/pgf/decoration/mark info/sequence number}) at (1,0pt);
	\coordinate (tangent orthogonal unit vector-\pgfkeysvalueof{/pgf/decoration/mark info/sequence number}) at (0pt,1);
	}
},
postaction=decorate]
\tikzstyle{use tangent}=[shift=(tangent point-#1),
	x=(tangent unit vector-#1),
	y=(tangent orthogonal unit vector-#1)]
\def\M{-0.4} 
\begin{tikzpicture}[xscale=0.5,yscale=0.7]
\draw[blue,domain=-pi:pi,tangent=0.01,tangent=0.05,tangent=0.1,tangent=0.3,tangent=0.45,tangent=0.52,tangent=0.73, tangent=0.96] plot[smooth] ({sin(\x r)},{4*cos(\x r)}) ;
\draw[blue,ultra thick,domain=-pi:pi/4] plot[smooth] ({sin(\x r)},{4*cos(\x r)}) ;
\draw[use tangent=1] (-1,0) -- (2,0);
\draw[use tangent=1,dashed] (-1,\M) -- (2,\M);
\draw[use tangent=1,<->,red] (-1,0) -- (-1,\M) node[above] {$\gamma$}; 
\draw[use tangent=3] (-2,0) -- (2,0);
\draw[use tangent=3,dashed] (-2,\M) -- (2,\M);
\draw[use tangent=4] (-2,0) -- (2,0);
\draw[use tangent=4,dashed] (-2,\M) -- (2,\M);
\draw[use tangent=5] (-2,0) -- (1,0);
\draw[use tangent=5,dashed] (-2,\M) -- (1,\M);
\draw[use tangent=6] (-1,0) -- (2,0);
\draw[use tangent=6,dashed] (-1,\M) -- (2,\M);
\draw[use tangent=7] (-2,0) -- (2,0);
\draw[use tangent=7,dashed] (-2,\M) -- (2,\M);
\draw[use tangent=8] (-2,0) -- (1,0);
\draw[use tangent=8,dashed] (-2,\M) -- (1,\M);
\fill[blue] (0,0) circle (3pt) node[right] {$\hat{x}$};
\end{tikzpicture}
\def\M2{-0.7}
\begin{tikzpicture}[xscale=0.5,yscale=0.7]
\draw[blue,domain=-pi:pi,tangent=0.01,tangent=0.05,tangent=0.1,tangent=0.3,tangent=0.45,tangent=0.52,tangent=0.73, tangent=0.96] plot[smooth] ({sin(\x r)},{4*cos(\x r)}) ;
\draw[blue,ultra thick,domain=-pi:pi/4] plot[smooth] ({sin(\x r)},{4*cos(\x r)}) ;
\draw[use tangent=1] (-1,0) -- (2,0);
\draw[use tangent=1,dashed,opacity=.2] (-1,\M2) -- (2,\M2);
\draw[use tangent=1,<->,blue] (1.5,0) -- (1.5,\M2) node[left] {$\gamma$}; 
\draw[use tangent=3] (-2,0) -- (2,0);
\draw[use tangent=3,dashed,opacity=.2] (-2,\M2) -- (2,\M2);
\draw[use tangent=4] (-2,0) -- (2,0);
\draw[use tangent=4,dashed] (-2,\M2) -- (2,\M2);
\draw[use tangent=5] (-2,0) -- (1,0);
\draw[use tangent=5,dashed,opacity=.2] (-2,\M2) -- (1,\M2);
\draw[use tangent=6] (-1,0) -- (2,0);
\draw[use tangent=6,dashed] (-1,\M2) -- (3,\M2);
\draw[use tangent=7,red,thick] (-3,0) -- (3,0);
\draw[use tangent=7,dashed,red,thick] (-3,\M2) -- (3,\M2);
\draw[use tangent=7,<->,red] (-0.15,-0.1) -- (-0.15,\M2) node[above right] {$\xi_i$}; 
\draw[use tangent=8] (-2,0) -- (1,0);
\draw[use tangent=8,dashed] (-3,\M2) -- (1,\M2);
\fill[blue] (0.83,0.65) circle (3pt) node[right] {$\hat{x}$};
\end{tikzpicture}
\caption{Illustration of the soft-margin approach on a toy example where random linear constraints are substituted for a nonlinear constraint (blue ellipsoid). {\em Left:} The hard-margin scenario program yields a solution $\hat{x}$ that is feasible for all $\theta\in\Theta$ with $V(\hat{x})=0$. However, the margin \textcolor{red}{$\gamma$} is rather small and cannot be increased without making Problem~\ref{pb:hardmargin} infeasible. Thus, the upper bound on $V(\hat{x})$ given by Theorem~\ref{thm:rad} is not really tight and, though the solution is perfect, we have a limited confidence in its performance. {\em Right: } By allowing the violation of one margin constraint (the red dashed line) with a slack of \textcolor{red}{$\xi_i$}, the soft-margin version yields a solution with a larger margin \textcolor{blue}{$\gamma$}, and thus with more confidence in its performance despite the increase in the empirical error $\hat{V}_{\gamma}(\hat{x})$. \label{fig:softmargin}}
\end{figure}
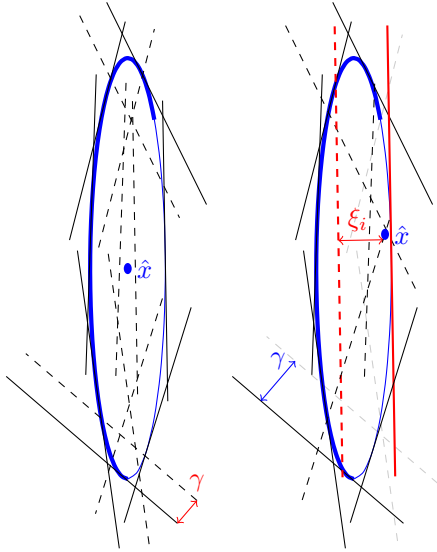

\section{Analysis} 
\label{sec:analysis}

In order to refine the analysis of the quality of the solution of Problems~\ref{pb:hardmargin}--\ref{pb:softmargin}, we introduce margin counterparts to the risks. These are based on a margin loss function $\ell_{\gamma} : \X\times \Theta\to[0,1]$ which measures the ability of $x$ to satisfy the margin constraint $f(x,\theta)\leq -\gamma$ for a single value of $\theta$. 
More precisely, we consider the piecewise linear margin loss function given by
\begin{equation}\label{eq:piecewiseloss}
	\ell_{\gamma}(x,\theta) = \begin{cases}
		1, & \mbox{if } f(x,\theta) \geq 0\\
		1 + \frac{f(x,\theta)}{\gamma}, &\mbox{if } f(x,\theta)\in (-\gamma, 0)\\
		0, & \mbox{if } f(x,\theta)\leq -\gamma 
		\end{cases}
\end{equation}
and define the risk at margin $\gamma$ as
$$
	V_{\gamma}(x) = \E _{\theta}\ell_{\gamma}(x,\theta), 
$$
and the empirical margin risk as 
$$
	\hat{V}_{\gamma}(x) = \frac{1}{N}\sum_{i=1}^N\ell_{\gamma}(x,\theta_i).
$$
In comparison with the standard definition of the risk, the margin risk $V_{\gamma}(x)$ based on the loss~\eqref{eq:piecewiseloss} also accounts for scenarios on which the constraint is satisfied with not enough margin. It is easy to verify that
\begin{equation}\label{eq:lossupperbound}
	\forall x\in\X,\theta\in\Theta,\quad \ell_{\gamma}(x,\theta) \geq \I{f(x,\theta) > 0}
\end{equation}
and therefore that $V_{\gamma}(x) \geq \E_{\theta}\I{f(x,\theta)>0} = V(x)$ always holds. Thus, by upper bounding the margin risk $V_{\gamma}(x)$, we will also bound the true probability of violation $V(x)$, and this can be done at a finer scale since $V_{\gamma}(x)$ depends more precisely on the values of $f(x,\theta)$ and not just on its sign as $V(x)$.

\subsection{Bounds on the probability of violation}
\label{sec:generalization}

Inspired by the statistical learning approach initiated by \cite{Koltchinskii02,Bartlett02}, the behavior of the margin risk is here analyzed in terms of the Rademacher complexity. 
\begin{definition}[Rademacher complexity] 
\label{def:radcomp}
For $N \in \mathbb{N}^*$,
let $\g{\theta}_N = \left( \theta_i  \right)_{1 \leq i \leq N}$
be an $N$-sample of independent copies of the random variable $\theta\in\Theta$, let
$\boldsymbol{\sigma}_N = \left ( \sigma_i \right )_{1 \leq i \leq N}$
be a sequence of independent random variables uniformly distributed in $\{-1,+1\}$. 
Let $\mathcal{F}$ be a class of real-valued functions on $\Theta$.
The {\em empirical Rademacher complexity} of $\mathcal{F}$ given $\g \theta_N$ is
$$
\hat{\mathcal{R}}_N \left ( \mathcal{F} \right )  = \E_{\g{\sigma}_N}\left[\left. \sup_{f \in \F} \frac{1}{N}\sum_{i=1}^N \sigma_i f \left ( \theta_i \right )\ \right| \g \theta_N \right].
$$
The {\em Rademacher complexity} of $\F$ is $\mathcal{R}_N(\F) = \E_{\g\theta_N} \hat{\mathcal{R}}_N \left ( \F \right )$.
\end{definition}

In particular, we have the following generic theorem, that applies similarly to the solution of Problem~\ref{pb:rand},~\ref{pb:hardmargin} or~\ref{pb:softmargin}.
\begin{theorem}\label{thm:rad}
For any $\gamma>0$ and $\delta\in(0,1)$, with probability at least $1-\delta$ on the random draw of $(\theta_i)_{1\leq i\leq N}$, the probability of violation is uniformly bounded for all $x\in\X$ by
$$
	V(x) \leq V_{\gamma}(x)\leq \hat{V}_{\gamma}(x) + \frac{2}{\gamma}\mathcal{R}_N(\mathcal{F}) + \sqrt{\frac{\log \frac{1}{\delta}}{2N}},
$$
where
\begin{equation}\label{eq:F}
	\F = \{f_x \in \R^{\Theta}: f_x(\theta) = f(x,\theta)  ,\ x\in\X \}.
\end{equation}
\end{theorem}
\begin{proof}
As already pointed out, the first inequality is a direct consequence of~\eqref{eq:lossupperbound}. 
The second inequality stems from Theorem~\ref{thm:general} in Appendix~\ref{sec:tools} applied to the class of loss functions
\begin{equation}\label{eq:lossclassrad}
	\L_{\gamma} = \{\ell_{\gamma,x} \in [0,1]^{\Theta}: \ell_{\gamma,x}(\theta) = \ell_{\gamma}(x,\theta),\ x\in\X \}, 
\end{equation}
and which guarantees that, with probability at least $1-\delta$, 
\begin{equation}\label{eq:boundproofg1}
	V_{\gamma}(x) \leq \hat{V}_{\gamma}(x) + 2\mathcal{R}_N(\mathcal{L}_{\gamma}) + \sqrt{\frac{\log \frac{1}{\delta}}{2N}}.
\end{equation}
Then, the Rademacher complexity of the loss class~\eqref{eq:lossclassrad} is bounded in terms of the class $\F$ using the contraction principle (Lemma~\ref{lem:contraction} in App.~\ref{sec:tools}) together with the fact that $\ell_{\gamma}$ can be written as
$$
	\ell_{\gamma}(x,\theta) = \varphi \circ f(x,\theta)
$$
with the function $\varphi(t) = \min\{1,\max\{0,1 + t/\gamma\}\}$ of Lipschitz constant equal to $1/\gamma$. This leads to
$$
	\mathcal{R}_N(\mathcal{L}_{\gamma}) \leq \frac{1}{\gamma} \mathcal{R}_N(\mathcal{F})
$$
and concludes the proof.
\end{proof}

Theorem~\ref{thm:rad} allows one to bound the risk in terms of the empirical risk with a confidence interval that depends on the complexity of the constraint functions induced by the domain $\X$. This is very much related to the approach of \cite{Alamo09}, which used the Vapnik-Chervonenkis (VC) dimension of 
$$
	\L=\{\ell_{x} \in [0,1]^{\Theta}: \ell_{x}(\theta) = \I{f(x,\theta) > 0},\ x\in\X \}
$$
to measure this complexity. Here, the Rademacher complexity offers a finer measure of capacity that also takes into account the magnitude of $f(x,\theta)$, whereas the VC-dimension of $\L$ only depends on its sign. However, the two can be related (see, e.g., \cite{Mohri18}): for a class $\L$ of finite VC-dimension $d_{VC}$, results in the flavor of those of \cite{Alamo09} can be recovered in our general approach by substituting $\L$ for the loss class $\L_{\gamma}$ in the proof above and bounding its Rademacher complexity by 
\begin{equation}\label{eq:vcdim}
	\mathcal{R}_N(\mathcal{L})\leq \sqrt{\frac{2 d_{VC}\log \frac{\mathrm{e} N}{d_{VC}}}{N}}.
\end{equation}

An important feature of Theorem~\ref{thm:rad} is that it holds uniformly over $\X$, meaning that it applies whether one truly solves Problems~\ref{pb:hardmargin}--\ref{pb:softmargin} or not, as long as the computed solution $\hat{x}$ lies in the domain $\X$. This is all the more relevant when considering nonconvex instances for which true solutions are difficult to obtain, and stands in contrast to the standard results discussed in the literature \cite{Calafiore10,Esfahani14,Campi18} that only apply to the optimizer of the scenario problem.

The influence of the margin $\gamma$ on the performance guarantees are clear from Theorem~\ref{thm:rad}: the confidence interval of the upper bound on the probability of violation decreases linearly with $1/\gamma$. Thus, a larger margin ensures better confidence, as illustrated by Fig.~\ref{fig:softmargin}. In the right plot of this Figure, the empirical error $\hat{V}_{\gamma}(x)$ increases by $\xi_i/\gamma N$ as one of the margin constraints is violated.\footnote{Note that  Problem~\ref{pb:softmargin} incurs a linear cost for errors $\xi_i$, whereas the loss function~\eqref{eq:piecewiseloss} saturates at $1$. Thus, in general, each error incurs a loss $\min(\xi_i/\gamma, 1)/N$ in $\hat{V}_{\gamma}(x)$, while in Fig.~\ref{fig:softmargin}, $\xi_i<\gamma$.} Meanwhile, the increase of the margin decreases the confidence interval and, overall, leads to a more favorable risk bound than in the left plot. Here, the thick part of the ellipsoid represents the points where tangent lines are drawn with high probability, while the thin part corresponds to random constraints that occur with low probability. Thus, the red constraint is a rare event and violating it does not incur a large increase of the probability of violation, which explains why a solution with more errors can lead to better guarantees in generalization. The low probability also means that there are fewer such constraints in the scenarios and that the cost of Problem~\ref{pb:softmargin} is minimized when violating the red constraint rather than multiple other constraints.  

In order to produce a performance guarantee, the Rademacher complexity appearing in Theorem~\ref{thm:rad} must be estimated.  The following theorem shows how to upper bound this complexity for the specific case of constraints functions as in~\eqref{eq:fminmax}. 
\begin{theorem}\label{thm:radboolean}
Let $\F$ be a function class as in~\eqref{eq:F} with $f$ of the form~\eqref{eq:fminmax}. Then,
$$
	\mathcal{R}_N(\mathcal{F}) \leq \sum_{k=1}^C \frac{(\prod_{j=k}^C \overline{\rho}_j)\overline{\varphi}_k\tau_k \Lambda_k}{\sqrt{N}},
$$
where $\tau_k = \sup_{\theta\in\Theta} \|\psi_k(\theta)\|$, $\Lambda_k =\sup_{x\in\X} \|\phi_k(x)\|$, $\overline{\rho}_1=1$, $\overline{\rho}_k$ for $k\geq 2$ and $\overline{\varphi}_k$ for $k\geq 1$ are the Lipschitz constants of the functions $\rho_k$ and $\varphi_k$, respectively.
\end{theorem} 
\begin{proof}
See Appendix~\ref{sec:theorem2}.
\end{proof}

Theorem~\ref{thm:radboolean} shows that the complexity, and thus the generalization performance via Theorem~\ref{thm:rad}, does not depend on the dimension of $x$ but rather on the size of the $\phi_k(x)$'s over $\X$ (as measured by $\Lambda_k$).\footnote{Also note that, due to the definition of $\Lambda_k$, $\X$ could be replaced by any (possibly discrete) subset $\X'\subset\X$ without any negative impact on the statistical guarantees.} This stands in contrast to most of the literature on scenario optimization, such as \cite{Calafiore10,Alamo09,Esfahani14}, where the dimension plays a major role. This can be seen in~\eqref{eq:calafiore2} for convex problems as studied in \cite{Calafiore10}, while for the approach of \cite{Alamo09}, the VC-dimension $d_{VC}$ in~\eqref{eq:vcdim} grows linearly with $d$ even for the most simple constraint function~\eqref{eq:fminmax} with $C=1$, $\varphi_1$ the identity and $f_1$ a linear function of $x$. Finally, when considering integer programming problems, the method proposed in \cite{Esfahani14} also leads to a bound linear in the number of integer variables.

However, there is nothing special about vectors $\phi_k(x)$ with small norms, i.e., closely gathered around the origin. Indeed, what really matters is the radius of the smallest ball enclosing the $\phi_k(x)$'s induced by all $x\in\X$, as made precise in the following corollary. Therefore, if the search domain $\X$ can be a priori reduced in size (for instance after the computation of a preliminary good guess of the solution), the radius could typically be made smaller, leading to better guarantees on the probability of violation.  
\begin{corollary}\label{cor:centered}
Let $\F$ be a function class as in~\eqref{eq:F} with $f$ of the form~\eqref{eq:fminmax}. Then, for any choice of centers $(\phi_{0,k})_{1\leq k \leq C} \in \prod_{k=1}^C \R^{n_k}$, 
$$
	\mathcal{R}_N(\mathcal{F}) \leq \sum_{k=1}^C \frac{(\prod_{j=k}^C \overline{\rho}_j)\overline{\varphi}_k\tau_k \tilde{\Lambda}_k}{\sqrt{N}},
$$
where $\tilde{\Lambda}_k = \sup_{x\in\X} \|\phi_k(x) - \phi_{0,k}\|$ and the other constants are as in Theorem~\ref{thm:radboolean}. 
\end{corollary}
\begin{proof}
Given a choice of $(\phi_{0,k})_{1\leq k \leq C}$, reformulate the $f_k$'s as
\begin{align*}
	f_k(x,\theta) &= \psi_k(\theta)^\top (\phi_k(x) - \phi_{0,k}) + (\eta_k(\theta) + \psi_k(\theta)^\top\phi_{0,k})\\
					&= \psi_k(\theta)^\top \tilde{\phi}_k(x) + \tilde{\eta}_k(\theta) 
\end{align*}
and note that in the proofs of Theorem~\ref{thm:radboolean} and Lemma~\ref{lem:radpseudolinear}, the definition of $\eta_k(\theta)$ does not play any role (as long as it does not depend on $x$). Therefore, they hold similarly with $\tilde{\phi}_k(x)$, $\tilde{\eta}_k(\theta)$ instead of $\phi_k(x)$, $\eta_k(\theta)$, and $\sup_{x\in\X} \|\phi_k(x)\|$ replaced by 
$$
	\sup_{x\in\X} \|\tilde{\phi}_{k}(x) \| = \sup_{x\in\X} \|\phi_k(x) - \phi_{0,k}\| = \tilde{\Lambda}_k.
$$ 
\end{proof}
Note that Corollary~\ref{cor:centered} provides an improvement over Theorem~\ref{thm:radboolean}, since it allows the choice $\phi_{0,k}=0$. 

The combination of Theorems~\ref{thm:rad} and \ref{thm:radboolean} (or Corollary~\ref{cor:centered}) leads to a bound on the probability of violation that converges to zero as $O(1/\sqrt{N})$, which, at first glance, might appear weaker than most results in the literature on scenario optimization, such as~\eqref{eq:calafiore1}--\eqref{eq:calafiore2}, which typically lead to bounds with a faster convergence rate in $O(d/N)$. However, a bound in $O(d/N)$ that improves upon Theorems~\ref{thm:rad}--\ref{thm:radboolean} requires that $d=O(\sqrt{N})$, which fails in high-dimensional regimes. This is all the more relevant that scenario optimization is mostly concerned with the non-asymptotic case which can offer practical solutions to difficult robust optimization problems with a finite and reasonable $N$, thus limiting the benefit of bounds in $O(d/N)$ to rather small dimensions in practice. This informal argument in favor of the proposed approach for high-dimensional problems, and even convex ones, will be made precise below when discussing the sample complexity. 

Regarding the linear dependence of the bounds with respect to $C$, it is reminiscent of the fact that the constraint function~\eqref{eq:fminmax} is built from a combination of several constraints/components. A similar effect can be observed in the results of \cite{Zhang15} for convex problems based on~\eqref{eq:fmax}.\footnote{The term $\prod_{j=k}^C \overline{\rho}_j$ is not considered here as it typically amounts to a product of values close to one, and is even exactly $1$ for all problems analyzed in~\cite{Zhang15}.}   

\subsection{Sample complexity}
\label{sec:samplecomplexity}

The sample complexity of a scenario problem is the minimal number of scenarios that are required to guarantee that, with high probability, the probability of violation does not exceed a predefined threshold. The following sample complexity estimate can be derived from the results in Theorems~\ref{thm:rad} and~\ref{thm:radboolean}.
\begin{corollary}
Let $f$ be of the form~\eqref{eq:fminmax}. Then, for any $N$ larger than or equal to the sample complexity
\begin{equation}\label{eq:refsamplecomplexityrad}
	N(\epsilon,\delta)= \frac{\left(\frac{2}{\gamma}\sum_{k=1}^C \tau_k\Lambda_k + \sqrt{\log\frac{1}{\delta^{1/2}}}\right)^2}{\epsilon^2},
\end{equation}
the probability of violation $V(x)$ of any $x\in\X$ is less than or equal to $\hat{V}_{\gamma}(x) + \epsilon$ with probability at least $1-\delta$.

In addition, if Problem~\ref{pb:hardmargin} is feasible, then all its solutions $\hat{x}$ satisfy $\hat{V}_{\gamma}(\hat{x}) =0$ and the bound $V(\hat{x})\leq \epsilon$ holds with probability at least $1-\delta$ for any $N\geq N(\epsilon,\delta)$. 
\end{corollary}

Though the sample complexity thus obtained is quadratic wrt. $1/\epsilon$, its numerical value can be less than that obtained with the standard approach for convex problems in high dimension. Indeed, as shown in \cite{Calafiore10}, for convex problems,~\eqref{eq:calafiore1}--\eqref{eq:calafiore2} translate into the sample complexity 
\begin{equation}\label{eq:samplecomplexitycalafiore}
	N_{\text{cvx}}(\epsilon,\delta) = \frac{2d + 2\log\frac{1}{\delta} }{\epsilon},
\end{equation}
which grows linearly with the dimension $d$. Therefore, it suffices that 
$$
	d>\frac{\left(\frac{2}{\gamma}\sum_{k=1}^C \tau_k\Lambda_k + \sqrt{\log\frac{1}{\delta^{1/2}}}\right)^2}{2\epsilon}
$$
to observe $N_{\text{cvx}}(\epsilon,\delta) > N(\epsilon,\delta)$. 
Consider for instance convex problems with $C=1$, $\tau_1\Lambda_1/\gamma = 2$, $\epsilon=0.03$ and $\delta=0.001$. Then, $N_{\text{cvx}}(\epsilon,\delta) > N(\epsilon,\delta)$ and the proposed approach improves upon \cite{Calafiore10} for all $d\geq 573$.

\subsection{Optimizing the margin}
\label{sec:optimmargin}

An insight from Theorem~~\ref{thm:rad} is that better generalization results from a larger margin $\gamma$. Therefore, it would be tempting to replace Problem~\ref{pb:hardmargin} by the following.
\begin{problem}[Max-margin scenario program]\label{pb:maxmargin}
Given a probability distribution of $\theta\in\Theta$, draw a sample of $N$ independent parameter values  $(\theta_i)_{1\leq i\leq N} \subset\Theta$ and solve
\begin{align*}
& \max_{x\in\X, \gamma>0} \gamma\\
\text{s.t. } & f(x,\theta_i) \leq -\gamma ,\quad i=1,\dots, N.
\end{align*}
\end{problem}
\smallskip
However, Theorem~\ref{thm:rad} only holds for a value of $\gamma$ fixed a priori and cannot be applied with the margin resulting from Problem~\ref{pb:maxmargin} and that depends on the random data $(\theta_i)_{1\leq i\leq N}$. Nontetheless, it can be extended to hold uniformly over $\gamma$ by following the proof of Theorem~5.9 in \cite{Mohri18}, with the addition of a small term in $O(\sqrt{\log\log_2(1/\gamma)})$.
\begin{theorem}\label{thm:gamma}
Given an upper bound  $\overline{\gamma}>0$ on the margin and $\delta\in(0,1)$, with probability at least $1-\delta$ on the random draw of $(\theta_i)_{1\leq i\leq N}$, the probability of violation is uniformly bounded for all $x\in\X$  by
\begin{align*}
	V(x) \leq V_{\gamma}(x)\leq \inf_{\gamma\in(0,\overline{\gamma}]} &\ \hat{V}_{\gamma}(x) + \frac{4}{\gamma}\mathcal{R}_N(\mathcal{F})  + \sqrt{\frac{\log \frac{1}{\delta}}{2N}} \\&\ + \sqrt{\frac{\log\log_2 \frac{2\overline{\gamma}}{\gamma}}{N}},
\end{align*}
where $\F = \{f_x \in \R^{\Theta}: f_x(\theta) = f(x,\theta)  ,\ x\in\X \}$.
\end{theorem}

Therefore, the value of $\gamma$ could be chosen a posteriori, in order to produce the best bound on the probability of violation. 
This could be done also with the standard scenario program of Problem~\ref{pb:rand} that returns a feasible point $\hat{x}$, possibly within the margin for some of the constraints. In this case, increasing the margin decreases the confidence interval of Theorem~\ref{thm:gamma}, while it increases the empirical margin risk $\hat{V}_{\gamma}(x)$, and the best trade-off between these two terms can be sought for.

\subsection{A posteriori bounds and regularization}
\label{sec:posteriori}

In cases where the distribution of $\theta$ is unknown or samples cannot be easily generated, the standard Monte Carlo approach to assess the probability of violation on an independent and large sample is not available. Then, a posteriori bounds can be used to estimate the probability of violation $V(x)$ of a solution $x$ computed with a fixed budget of $N$ observations $(\theta_i)_{1\leq i\leq N}$ by the empirical error $\hat{V}(x)$ computed on the same data. 

While the bounds on $V(x)$ in the theorems above could be applied for this purpose, they rely on a worst-case strategy in order to be computable a priori or to yield sample complexity estimates. Instead, refined a posteriori bounds can be computed on the basis of the {\em observed value} of the sample $(\theta_i)_{1\leq i\leq N}$ using the {\em empirical} Rademacher complexity. In this case, Theorem~\ref{thm:rad} holds with only minor changes in the constants,   
and Theorem~\ref{thm:radboolean} can be reshaped with $\sqrt{\sum_{i=1}^N \|\psi_k(\theta_i)\|^2} / N$ substituted for $\tau_k/\sqrt{N}$.
In addition, the constants $\Lambda_k$ can also be replaced by empirical versions that measure quantities computed for the specific solution $\hat{x}$ returned by the algorithm rather than global upper bounds. All these modifications are summarized in the following theorem. 
\begin{theorem}\label{thm:posterior}
Let $f$ be of the form~\eqref{eq:fminmax}. Then, for any choice of centers $(\phi_{0,k})_{1\leq k \leq C} \in \prod_{k=1}^C \R^{n_k}$  and $\delta\in(0,1)$, with probability at least $1-\delta$ on the random draw of $(\theta_i)_{1\leq i\leq N}$, the probability of violation is uniformly bounded for all $x\in\X$ by
\begin{align*}
	V(x) 
	&\leq \hat{V}_{\gamma}(x) 
	\\&\quad		+ \frac{2 \overline{\Lambda}(x) \sum_{k=1}^C \left(\prod_{j=k}^C \overline{\rho}_j\right)\overline{\varphi}_k \sqrt{\sum_{i=1}^N \|\psi_k(\theta_i)\|^2} }{\gamma N}  
		\\&\quad	 + 3\sqrt{\frac{\log \frac{6}{\delta}}{2N}} + 3\sqrt{\frac{\log\log_2 2 \overline{\Lambda}(x) }{N}}
\end{align*}
with $\overline{\Lambda}(x) = \max\{1,\|\phi_1(x) -\phi_{0,1}\| ,\dots, \|\phi_C(x) -\phi_{0,C}\| \}$. 
\end{theorem}
\begin{proof}
See Appendix~\ref{sec:proofposterior}.
\end{proof}
With Theorem~\ref{thm:posterior} at hand, it is now possible to envision a regularized scenario approach, in which we compute $\hat{x}$ as the solution that minimizes $\overline{\Lambda}(x)$. 
According to  Theorem~\ref{thm:posterior}, this $\hat{x}$ should enjoy better generalization performance than others with similar empirical error $\hat{V}_{\gamma}(x)$. In particular, if we compute $\hat{x}$ using the hard margin constraints of  Problem~\ref{pb:hardmargin}, then regularization offers a sound tie-breaking rule to choose among all feasible $x$ with $\hat{V}_{\gamma}(x)=0$. 

In practice, the centers $(\phi_{0,k})_{1\leq k \leq C}$ appearing in $\overline{\Lambda}(x)$ must be fixed in advance and can typically be chosen as the centers of the sets $\{\phi_k(x) : x\in\X\}$. Alternatively, one could compute the centers as $\phi_{0,k}=\phi_k(\hat{x})$ for a preliminary guess $\hat{x}$ of the solution obtained on an independent sample of scenarios.

\subsection{Fast rates}
\label{sec:fastrates}

As already mentioned, most of the literature on scenario optimization concentrates on bounds on the probability of violation that are in $O(1/N)$ rather than $O(1/\sqrt{N})$.  
Rademacher complexity-based bounds as proposed here can be extended using localization arguments \cite{Bartlett05,Srebro10} to yield bounds in $O(1/N)$ that are still independent of the dimension $d$ of $x$. However, we retain from giving the details for two reasons. First, such extensions typically apply only to the minimizer $\hat{x}$ of $\hat{V}_{\gamma}(x)$, which can be difficult to compute in nonconvex cases. Second, this also introduces rather large constants and the interest of the resulting bounds remains questionable in the non-asymptotic case (say, for $N<10^7$). In statistical learning, another path that leads to fast rate bounds is based on the use of covering numbers instead of the Rademacher complexity and relies on relative deviation bounds, as in \cite{Cortes21}. Following these ideas, we obtain the risk bound below for the margin-based scenario approach.\footnote{For the sake of clarity, we only expose the result for the hard-margin version (or zero-error case) here, but a more general bound including errors is proved in Appendix~\ref{sec:proofcovering}.} 
\begin{theorem}\label{thm:covering} 
Let the constraints be of the form~\eqref{eq:fminmax}. Then, for any $\delta\in(0,1)$, with probability at least $1-\delta$ on the random draw of $(\theta_i)_{1\leq i\leq N}$, the probability of violation is bounded, for any feasible $x$ in Problem~\ref{pb:hardmargin}, by
\begin{align*}
	V(x) \leq 4\frac{144\sum_{k=1}^C M_k^2 \log (60 M_k N) + \log\frac{4}{\delta}}{N},
\end{align*}
where 
$$
	M_k = \frac{2^{p_k}  \overline{\varphi}_k\tau_k\Lambda_k \prod_{j=k}^C \overline{\rho}_j}{\gamma},
$$
$p_k$ is the number of binary operations $g_j$ with $j=k,\dots,C$ that are sums or differences, and all other constants are as in Theorem~\ref{thm:radboolean}. 
\end{theorem}
\begin{proof}
See Appendix~\ref{sec:proofcovering}.
\end{proof}

Again, this result does not explicitly depend on $d$, and shows an improved convergence rate of $O(\log (N)/N)$. This can be transformed into a sample complexity estimate that grows in $O(\log(\epsilon)/\epsilon)$. 
Yet, the constants appearing in the above are larger than in the previous results based on the Rademacher complexity, which makes the improvement in terms of the convergence rate visible mostly for $N>10^5$. 

\section{Optimization problems and margin complexity}
\label{sec:optim}

First of all, notice that all results presented above hold similarly for random optimization problems of the form:
\begin{align}\label{eq:scenariooptim}
&\min_{x\in\X} J(x) \\
\text{s.t. } & f(x,\theta_i) \leq -\gamma ,\quad i=1,\dots, N.\nonumber
\end{align}
However, introducing the margin can affect the objective value $J(\hat{x})$ of the computed solution $\hat{x}$. Since the maximization of the margin as proposed in Sect.~\ref{sec:optimmargin} and the minimization of $J(x)$ play against each other, one would have to find a suitable trade-off by tuning $\lambda>0$ and solving 
\begin{align*}
&\min_{x\in\X, \gamma>0} J(x) - \lambda\gamma\\
\text{s.t. } & f(x,\theta_i) \leq -\gamma ,\quad i=1,\dots, N.
\end{align*}

Another point of view arises when one considers the following question: given a fixed budget of $N$ scenarios, how can we minimize $J(x)$ while ensuring that $P^N\{V(x)\leq \epsilon\}\geq 1-\delta$? 
Clearly, the smallest $J(\hat{x})$ will result from the smallest choice of margin $\gamma$ in~\eqref{eq:scenariooptim}. This leads to the concept of {\em margin complexity}, i.e., the smallest margin $\gamma$ such that $P^N\{V(x)\leq \epsilon\}\geq 1-\delta$ for given $N$, $\epsilon$ and $\delta$. For constraint functions as in~\eqref{eq:fminmax} and the hard-margin scenario program~\eqref{eq:scenariooptim}, the margin complexity can be estimated using Theorems~\ref{thm:rad}--\ref{thm:radboolean} as  
$$
	\gamma(N,\epsilon,\delta) = 
		\frac{2\sum_{k=1}^C (\prod_{j=k}^C \overline{\rho}_j)\overline{\varphi}_k\tau_k \Lambda_k}{\epsilon\sqrt{N} - \sqrt{\log \frac{1}{\delta^{1/2}}}} .
$$
The procedure to handle an optimization problem with a fixed budget can thus be formulated as: solve~\eqref{eq:scenariooptim} with $\gamma=\gamma(N,\epsilon,\delta)$ given above. 
Then, we have the guarantee that the solution $\hat{x}$ satisfies $P^N\{V(\hat{x})\leq \epsilon\}\geq 1-\delta$.  

\section{Conclusions}
\label{sec:conclusions}

The paper introduced the notion of margin in the scenario approach to robust optimization and analyzed the statistical performance of the resulting random programs using tools from statistical learning theory, and more specifically the Rademacher complexity. This allowed for the derivation of a priori bounds on the probability of violation and sample complexities in which the dependence on the problem dimension could be removed. 
These results hold only for a specific class of constraint functions. But this class includes polynomials as particular cases and subsumes typical classes previously considered in the literature. In addition, the results derived in this paper apply similarly to both convex and nonconvex problems.

Future work will consider further generalizations of the class of problems for which such results can be derived. While the main focus of the paper was on feasibility problems and the probability of violation, the impact of the margin on the cost function for minimization problems should also be investigated further. 
More generally, this paper also encourages the use of techniques from learning theory that have been overlooked so far in the  context of scenario optimization.

\bibliographystyle{IEEEtran}

\begin{thebibliography}{}
\providecommand{\url}[1]{#1}
\csname url@samestyle\endcsname
\providecommand{\newblock}{\relax}
\providecommand{\bibinfo}[2]{#2}
\providecommand{\BIBentrySTDinterwordspacing}{\spaceskip=0pt\relax}
\providecommand{\BIBentryALTinterwordstretchfactor}{4}
\providecommand{\BIBentryALTinterwordspacing}{\spaceskip=\fontdimen2\font plus
\BIBentryALTinterwordstretchfactor\fontdimen3\font minus
  \fontdimen4\font\relax}
\providecommand{\BIBforeignlanguage}[2]{{%
\expandafter\ifx\csname l@#1\endcsname\relax
\typeout{** WARNING: IEEEtran.bst: No hyphenation pattern has been}%
\typeout{** loaded for the language `#1'. Using the pattern for}%
\typeout{** the default language instead.}%
\else
\language=\csname l@#1\endcsname
\fi
#2}}
\providecommand{\BIBdecl}{\relax}
\BIBdecl

\end{thebibliography}


\begin{thebibliography}{10}
\providecommand{\url}[1]{#1}
\csname url@samestyle\endcsname
\providecommand{\newblock}{\relax}
\providecommand{\bibinfo}[2]{#2}
\providecommand{\BIBentrySTDinterwordspacing}{\spaceskip=0pt\relax}
\providecommand{\BIBentryALTinterwordstretchfactor}{4}
\providecommand{\BIBentryALTinterwordspacing}{\spaceskip=\fontdimen2\font plus
\BIBentryALTinterwordstretchfactor\fontdimen3\font minus
  \fontdimen4\font\relax}
\providecommand{\BIBforeignlanguage}[2]{{%
\expandafter\ifx\csname l@#1\endcsname\relax
\typeout{** WARNING: IEEEtran.bst: No hyphenation pattern has been}%
\typeout{** loaded for the language `#1'. Using the pattern for}%
\typeout{** the default language instead.}%
\else
\language=\csname l@#1\endcsname
\fi
#2}}
\providecommand{\BIBdecl}{\relax}
\BIBdecl

\bibitem{Tempo13}
R.~Tempo, G.~Calafiore, and F.~Dabbene, \emph{Randomized algorithms for
  analysis and control of uncertain systems: with applications}.\hskip 1em plus
  0.5em minus 0.4em\relax Springer, 2013.

\bibitem{Campi08}
M.~Campi and S.~Garatti, ``The exact feasibility of randomized solutions of
  uncertain convex programs,'' \emph{SIAM Journal on Optimization}, vol.~19,
  no.~3, pp. 1211--1230, 2008.

\bibitem{Alamo09}
T.~Alamo, R.~Tempo, and E.~Camacho, ``Randomized strategies for probabilistic
  solutions of uncertain feasibility and optimization problems,'' \emph{IEEE
  Transactions on Automatic Control}, vol.~54, no.~11, pp. 2545--2559, 2009.

\bibitem{Calafiore10}
G.~Calafiore, ``Random convex programs,'' \emph{SIAM Journal on Optimization},
  vol.~20, no.~6, pp. 3427--3464, 2010.

\bibitem{Esfahani14}
P.~Esfahani, T.~Sutter, and J.~Lygeros, ``Performance bounds for the scenario
  approach and an extension to a class of non-convex programs,'' \emph{IEEE
  Transactions on Automatic Control}, vol.~60, no.~1, pp. 46--58, 2014.

\bibitem{Campi18b}
M.~Campi and S.~Garatti, ``Wait-and-judge scenario optimization,''
  \emph{Mathematical Programming}, vol. 167, no.~1, pp. 155--189, 2018.

\bibitem{Campi18}
M.~Campi, S.~Garatti, and F.~Ramponi, ``A general scenario theory for nonconvex
  optimization and decision making,'' \emph{IEEE Transactions on Automatic
  Control}, vol.~63, no.~12, pp. 4067--4078, 2018.

\bibitem{Garatti22}
S.~Garatti and M.~Campi, ``Risk and complexity in scenario optimization,''
  \emph{Mathematical Programming}, vol. 191, pp. 243--279, 2022.

\bibitem{Calafiore06}
G.~Calafiore and M.~Campi, ``The scenario approach to robust control design,''
  \emph{IEEE Transactions on Automatic Control}, vol.~51, no.~5, pp. 742--753,
  2006.

\bibitem{Koltchinskii02}
V.~Koltchinskii and D.~Panchenko, ``Empirical margin distributions and bounding
  the generalization error of combined classifiers,'' \emph{The Annals of
  Statistics}, vol.~30, no.~1, pp. 1--50, 2002.

\bibitem{Bartlett02}
P.~Bartlett and S.~Mendelson, ``Rademacher and {G}aussian complexities: Risk
  bounds and structural results,'' \emph{Journal of Machine Learning Research},
  vol.~3, pp. 463--482, 2002.

\bibitem{Zhang15}
X.~Zhang, S.~Grammatico, G.~Schildbach, P.~Goulart, and J.~Lygeros, ``On the
  sample size of random convex programs with structured dependence on the
  uncertainty,'' \emph{Automatica}, vol.~60, pp. 182--188, 2015.

\bibitem{Mohri18}
M.~Mohri, A.~Rostamizadeh, and A.~Talwalkar, \emph{Foundations of Machine
  Learning}, 2nd~ed.\hskip 1em plus 0.5em minus 0.4em\relax The MIT Press,
  2018.

\bibitem{Bartlett05}
P.~Bartlett, O.~Bousquet, and S.~Mendelson, ``Local {R}ademacher
  complexities,'' \emph{The Annals of Statistics}, vol.~33, no.~4, pp.
  1497--1537, 2005.

\bibitem{Srebro10}
N.~Srebro, K.~Sridharan, and A.~Tewari, ``Smoothness, low noise and fast
  rates,'' in \emph{Advances in Neural Information Processing Systems},
  vol.~23, 2010.

\bibitem{Cortes21}
C.~Cortes, M.~Mohri, and A.~Suresh, ``Relative deviation margin bounds,'' in
  \emph{ICML}, 2021, pp. 2122--2131.

\bibitem{Ledoux91}
M.~Ledoux and M.~Talagrand, \emph{Probability in Banach Spaces: Isoperimetry
  and Processes}.\hskip 1em plus 0.5em minus 0.4em\relax Springer-Verlag,
  Berlin, 1991.

\bibitem{Zhang02}
T.~Zhang, ``Covering number bounds of certain regularized linear function
  classes,'' \emph{Journal of Machine Learning Research}, vol.~2, pp. 527--550,
  2002.

\bibitem{Lauer20}
F.~Lauer, ``Error bounds for piecewise smooth and switching regression,''
  \emph{IEEE Transactions on Neural Networks and Learning Systems}, vol.~31,
  no.~4, pp. 1183--1195, 2020.

\end{thebibliography}

\appendices

\section{Tools for Rademacher complexities}
\label{sec:tools}

The following general theorem is a classical tool from statistical learning theory, see, e.g., Theorem~3.3 in \cite{Mohri18}.
\begin{theorem}
\label{thm:general}
Let $\mathcal{L}$ be a class of functions from $\Theta$ into $[0,1]$ and $(\theta_i)_{1\leq i\leq N}$ be a sequence of independent copies of the random variable $\theta\in\Theta$.  Then, for a fixed $\delta \in (0, 1)$,
with probability at least $1 - \delta$, each of the following holds uniformly over all $\ell \in \mathcal{L}$,
\begin{align*}
&\E_{\theta}  \ell(\theta) \leq \frac{1}{N}\sum_{i=1}^N \ell(\theta_i) + 2 \mathcal{R}_N \left ( \mathcal{L} \right )
+ \sqrt{\frac{\log \frac{1}{\delta} }{2N}} ,\\
&\E_{\theta}  \ell(\theta) \leq \frac{1}{N}\sum_{i=1}^N \ell(\theta_i) + 2 \hat{\mathcal{R}}_N \left ( \mathcal{L} \right )
+ 3\sqrt{\frac{\log \frac{2}{\delta} }{2N}} .
\end{align*}
\end{theorem}

We recall the contraction principle for Rademacher complexities, found in Theorem~4.12 of~\cite{Ledoux91} or Lemma~5.7 of~\cite{Mohri18}.
\begin{lemma}\label{lem:contraction}
If the function $\varphi: \R\rightarrow \R$ is $\overline{\varphi}$-Lipschitz, i.e., if $\forall (u,v)\in \R^2,\ |\varphi(u) - \varphi(v)| \leq \overline{\varphi} |u-v|$, 
then, for the class $\varphi \circ \F=\{\varphi \circ f : f \in\F\}$, 
$$
	\hat{\mathcal{R}}_N ( \varphi \circ \F ) \leq \overline{\varphi}\cdot\hat{\mathcal{R}}_N ( \F ) .
$$
\end{lemma}

\section{Proof of Theorem~\ref{thm:radboolean}}
\label{sec:theorem2}

The proof of Theorem~\ref{thm:radboolean} relies on the following bound on the Rademacher complexity of a class of functions $f_k$ as in~\eqref{eq:fk}.
\begin{lemma}\label{lem:radpseudolinear}
Given functions $\psi:\Theta\times \R^{n}$, $\phi:\Theta\to\R^n$ and $\eta:\Theta\to\R$, the empirical Rademacher complexity of the class $\F = \{f_x \in \R^{\Theta}: f_x(\theta) = \psi(\theta)^\top \phi(x) + \eta(\theta),\ x\in\X \}$ given $(\theta_i)_{1\leq i\leq N}$ is bounded by
$$
	\hat{\mathcal{R}}_N(\mathcal{F}) \leq \frac{\Lambda\sqrt{\sum_{i=1}^N \|\psi(\theta_i)\|^2}}{N} \leq \frac{\tau \Lambda}{\sqrt{N}}, 
$$ 
where $\tau = \sup_{\theta\in\Theta} \|\psi(\theta)\|$ and $\Lambda = \sup_{x\in\X}\|\phi(x)\|$. 
\end{lemma}
\begin{proof}
Using the subadditivity of the supremum, we have
\begin{align*}
	\hat{\mathcal{R}}_N(\mathcal{F}) &=\E\sup_{x\in\X} \frac{1}{N}\sum_{i=1}^N \sigma_i( \psi(\theta_i)^\top \phi(x) + \eta(\theta_i) )\\
	&\leq \E\sup_{x\in\X} \frac{1}{N}\sum_{i=1}^N \sigma_i \psi(\theta_i)^\top \phi(x) + \E\sup_{x\in\X} \frac{1}{N}\sum_{i=1}^N \sigma_i \eta(\theta_i),
\end{align*}
where
\begin{align*}
	\E\sup_{x\in\X} \frac{1}{N}\sum_{i=1}^N \sigma_i \eta(\theta_i)
	&=\E \frac{1}{N}\sum_{i=1}^N \sigma_i \eta(\theta_i)\\
	&=  \frac{1}{N}\sum_{i=1}^N \eta(\theta_i)\E \sigma_i = 0,
\end{align*}
since the expectation is computed only with respect to the Rademacher variables $\sigma_i$ that are centered ($\E\sigma_i = 0$). %
We can then follow the standard path for linear classes \cite{Bartlett02}, with the addition of the mappings $\psi$ and $\phi$: 
\begin{align*}
	\hat{\mathcal{R}}_N(\mathcal{F}) & \leq \E\sup_{x\in\X} \frac{1}{N}\sum_{i=1}^N \sigma_i \psi(\theta_i)^\top \phi(x) \\
	&= \frac{1}{N} \E\sup_{x\in\X} \left(\sum_{i=1}^N \sigma_i \psi(\theta_i)\right)^\top \phi(x) \\
	&\leq \frac{1}{N} \E\sup_{x\in\X} \left\|\sum_{i=1}^N \sigma_i \psi(\theta_i)\right\|\, \| \phi(x)\| \\
	&= \frac{\Lambda}{N} \E \left\|\sum_{i=1}^N \sigma_i \psi(\theta_i)\right\|,
\end{align*}
in which Jensen's inequality ensures that 
\begin{align*}
	\E \left\|\sum_{i=1}^N \sigma_i \psi(\theta_i)\right\| &= \E\sqrt{\left\|\sum_{i=1}^N \sigma_i \psi(\theta_i)\right\|^2}\\
	&\leq \sqrt{\E \left\|\sum_{i=1}^N \sigma_i \psi(\theta_i)\right\|^2}\\
	&= \sqrt{\E \sum_{i=1}^N \sum_{j=1}^N \sigma_i\sigma_j \psi(\theta_i)^\top\psi(\theta_j)}\\
	&= \sqrt{\sum_{i=1}^N  \|\psi(\theta_i)\|^2 + \E \sum_{j\neq i} \sigma_i\sigma_j \psi(\theta_i)^\top\psi(\theta_j)}\\
	&= \sqrt{\sum_{i=1}^N  \|\psi(\theta_i)\|^2},
\end{align*}
where we used the independence of the $\sigma_i$'s and the fact that $\E\sigma_i=0$ to remove the terms with $j\neq i$. Gathering the results and using $\sqrt{\sum_{i=1}^N  \|\psi(\theta_i)\|^2}\leq \sqrt{N}\tau$ complete the proof. 
\end{proof}

The proof of Theorem~\ref{thm:radboolean} also requires the introduction of a few function classes: for all $k\in\{1,\dots,C\}$, 
\begin{align*}
	\F_k &= \{f_x \in \R^{\Theta}: f_x(\theta) = f_k(x,\theta)  ,\ x\in\X \}
\end{align*}
and, for all $k\in\{2,\dots,C\}$,
\begin{align*}
	\G_k(\mathcal{U},\mathcal{V}) 
	&= \mathcal{U}+\mathcal{V},\ \mathcal{U} - \mathcal{V},\ \max(\mathcal{U},\mathcal{V}),\ \text{or } \min(\mathcal{U},\mathcal{V}),
\end{align*}
in accordance with $g_k$. 
With these classes at hand, we can rewrite the function class of interest, $\F$, in a recursive manner:
\begin{align}\label{eq:recursiveF}
	\F &= \F^C \\
	\F^1 &= \varphi_1\circ \F_1 \nonumber\\
	\F^{k+1} &= \rho_{k+1}\circ \G_{k+1}( \F^{k} ,\ \varphi_{k+1}\circ\F_{k+1}).\nonumber
\end{align}
Then, the proof works by induction over the number of components $C$. 

First, for $C=1$, $f(x,\theta) = \varphi_1\circ f_1(x,\theta)$ and $\F = \varphi_1\circ\F_1$. In this case, the contraction principle of Lemma~\ref{lem:contraction} in Appendix~\ref{sec:tools} yields
$$
	\mathcal{R}_N(\mathcal{F}) \leq \overline{\varphi}_1 \cdot \mathcal{R}_N(\mathcal{F}_1) .
$$
Then, Lemma~\ref{lem:radpseudolinear}  gives $\mathcal{R}_N(\mathcal{F}_1)\leq \tau_1\Lambda_1/\sqrt{N}$ and the result follows. 

Assume now that the statement holds for $C$ components, i.e., for $\F = \F^{C}$:
\begin{equation}\label{eq:recursiveass}
	\mathcal{R}_N(\mathcal{F}^C) \leq \sum_{k=1}^C \frac{(\prod_{j=k}^C \overline{\rho}_j)\overline{\varphi}_k\tau_k \Lambda_k}{\sqrt{N}}.
\end{equation}
Then, we can show that it also holds for $\F^{C+1}$. To see this, first apply the contraction principle again to obtain
\begin{align*}
	\mathcal{R}_N(\mathcal{F}^{C+1}) &= \mathcal{R}_N(\rho_{C+1}\circ \G_{C+1}( \F^{C} ,\ \varphi_{C+1}\circ\F_{C+1})) 
	\\&\leq \overline{\rho}_{C+1}\mathcal{R}_N( \G_{C+1}( \F^{C} ,\ \varphi_{C+1}\circ\F_{C+1})) .
\end{align*}
Then, it remains to show that 
\begin{align}\label{eq:proofGC+1}
	\mathcal{R}_N(\G_{C+1}(\F^{C}, \varphi_{C+1}\circ\F_{C+1})) \leq &\ \mathcal{R}_N(\F^{C})\\&\ + \mathcal{R}_N(\varphi_{C+1}\circ\F_{C+1})\nonumber
\end{align}
for all choices of the binary operator $g_{C+1}$ to conclude using~\eqref{eq:recursiveass} and $\mathcal{R}_N(\varphi_{C+1}\circ\F_{C+1}) \leq \overline{\varphi}_{C+1}\cdot\mathcal{R}_N(\F_{C+1})$ with, by Lemma~\ref{lem:radpseudolinear}, $\mathcal{R}_N(\F_{C+1})\leq \tau_{C+1}\Lambda_{C+1}/\sqrt{N}$.

If $g_{C+1}$ is the addition, then a basic property of Rademacher complexities (inherited from the subadditivity of the supremum and the linearity of the expectation), namely that $\mathcal{R}_N(\mathcal{U}+\mathcal{V})\leq \mathcal{R}_N(\mathcal{U}) + \mathcal{R}_N(\mathcal{V})$, suffices to prove~\eqref{eq:proofGC+1}. 

Since the random variables $\sigma_i$ and  $-\sigma_i$ share the same distribution, we have $\mathcal{R}_N(-\mathcal{V}) = \mathcal{R}_N(\mathcal{V})$ and~\eqref{eq:proofGC+1} is also proved for the case $\G_{C+1}(\mathcal{U} , \mathcal{V})=\mathcal{U} - \mathcal{V}$.

For the min and max operators, we can follow Lemma~9.1 in \cite{Mohri18} and rewrite $g_{C+1}(a,b) = \frac{1}{2}(a+b + s |a-b|)$ with $s=1$ for the $\max$ and $s=-1$ for the $\min$. Then, 
$$
	\mathcal{R}_N(\G_{C+1}(\mathcal{U},\mathcal{V})) \leq \frac{1}{2}\left[ \mathcal{R}_N( \mathcal{U}+\mathcal{V}) + \mathcal{R}_N(w\circ(\mathcal{U}-\mathcal{V})) \right]
$$
for the function $w : t\mapsto s|t|$ of Lipschitz constant equal to 1. Thus, using the contraction principle again and the results of the two previous cases, we obtain 
\begin{align*}
	\mathcal{R}_N(\G_{C+1}(\mathcal{U},\mathcal{V})) &\leq \frac{1}{2} \left[\mathcal{R}_N( \mathcal{U})+\mathcal{R}_N(\mathcal{V})\right. \\
		&\qquad \left. + \mathcal{R}_N(\mathcal{U})+\mathcal{R}_N(\mathcal{V})\right]\\
		& =\mathcal{R}_N(\mathcal{U})+\mathcal{R}_N(\mathcal{V})
\end{align*}
and thus~\eqref{eq:proofGC+1}.

Therefore, for any choice of $g_{C+1}$, 
\begin{align*}
	\mathcal{R}_N&(\mathcal{F}^{C+1}) \leq \overline{\rho}_{C+1}\left(\mathcal{R}_N(\F^{C}) + \mathcal{R}_N(\varphi_{C+1}\circ\F_{C+1}) \right)\\
	&\leq \overline{\rho}_{C+1}\left(\sum_{k=1}^C \frac{(\prod_{j=k}^C \overline{\rho}_j)\overline{\varphi}_k\tau_k \Lambda_k}{\sqrt{N}} + \overline{\varphi}_{C+1} \frac{\tau_{C+1}\Lambda_{C+1}}{\sqrt{N}}\right)\\
	&= \sum_{k=1}^{C+1} \frac{(\prod_{j=k}^{C+1} \overline{\rho}_j)\overline{\varphi}_k\tau_k \Lambda_k}{\sqrt{N}} 
\end{align*}
and Theorem~\ref{thm:radboolean} is proved by induction for all $C\geq 1$.

\section{Proof of Theorem~\ref{thm:posterior}}
\label{sec:proofposterior}

Define 
$$
	R=\sum_{k=1}^C \left(\prod_{j=k}^C \overline{\rho}_j\right)\overline{\varphi}_k \sqrt{\sum_{i=1}^N \|\psi_k(\theta_i)\|^2}
$$
and, for any $\alpha>9/2$, let $p=2\alpha/9$, $\zeta(p)$ denote the Riemann zeta function and
\begin{equation}\label{eq:proofepsilon}
	\epsilon = 3\sqrt{\frac{\log \frac{2+2\zeta(p)}{\delta} }{ 2N}},
\end{equation}
$$
	\epsilon_{x} = \epsilon + \sqrt{\frac{\alpha\log\log_2 2\overline{\Lambda}(x) }{N}} .
$$ 
We will show that, for any $\delta\in(0,1)$, 
\begin{equation}\label{eq:proofthm3}
	P^N\left\{\forall x\in\X,\ V_{\gamma}(x) \leq \hat{V}_{\gamma}(x) + \frac{4 \overline{\Lambda}(x)R}{\gamma N} + \epsilon_{x}\right\}\geq 1-\delta ,
\end{equation}
from which Theorem~\ref{thm:posterior} is derived by choosing $\alpha=9$ and noting that $\zeta(2)=\pi^2/6<2$. 

For any subset $\X_j\subseteq\X$, let $\F(\X_j) = \{f_x\in\R^{\Theta} : f_x(\theta) = f(x,\theta),\ x\in\X_j\}$. 
Following the proof of Theorem~\ref{thm:rad} while using the second inequality of Theorem~\ref{thm:general} instead of the first yields
\begin{align}\nonumber
	P^N&\!\left\{\exists x\in\X_j, V_{\gamma}(x)> \hat{V}_{\gamma}(x) + \frac{2}{\gamma}\hat{\mathcal{R}}_N(\mathcal{F}(\X_j)) + 3\sqrt{\frac{\log \frac{2}{\delta}}{2N}} \right\} \\
	&\quad \leq \delta.\label{eq:proofposterior1}
\end{align}
Then, Theorem~\ref{thm:radboolean} and Corollary~\ref{cor:centered} can be reshaped to yield
\begin{align}\label{eq:proofemprad}
	\hat{\mathcal{R}}_N(\mathcal{F}(\X_j)) \leq \sum_{k=1}^C & \left(\prod_{k'=k}^C \overline{\rho}_{k'}\right)\overline{\varphi}_k  \sup_{x\in\X_j} \|\phi_k(x) - \phi_{0,k}\| \\
		& \times\frac{\sqrt{\sum_{i=1}^N \|\psi_k(\theta_i)\|^2}}{N} . \nonumber .
\end{align}
To see this, note that their proofs hold verbatim with the empirical version of the Rademacher complexities and all $\tau_k/\sqrt{N}$ replaced by $\sqrt{\sum_{i=1}^N \|\psi_k(\theta_i)\|^2}/N$, as given by Lemma~\ref{lem:radpseudolinear}. 

Now, let $l_j = 2^j$ and, for any integer $j\geq 1$, $\epsilon_{j} = \epsilon + \sqrt{\alpha \log\log_2 (l_j) /N}$ and 
\begin{equation}\label{eq:proofXj}
	\X_{j} = \{x\in\X : \Lambda(x) = \max_{k\in\{1,\dots,C\}}\|\phi_k(x) - \phi_{0,k}(x)\| \leq l_j \}
\end{equation}
with $\overline{\Lambda}(x) = \max\{1,\Lambda(x)\}$. 
For any $x\in\X$, either $x\in\X_{0}$ with $\overline{\Lambda}(x) = 1 >1/2$ and $\epsilon_{x} = \epsilon$, or there is a $j\geq 1$ such that $x\in\X_{j}$ with $\overline{\Lambda}(x) = \Lambda(x)$, $l_j\geq \Lambda(x) > l_{j-1} = l_j/2$ and $\epsilon_{x} > \epsilon_{j}$. 
Thus,
\begin{align*}
	P^N\!&\left\{\exists x\in\X,\ V_{\gamma}(x) > \hat{V}_{\gamma}(x) + \frac{4 \overline{\Lambda}(x)R}{\gamma N} + \epsilon_{x}\right\}\\
	&= P^N\!\left\{\exists j\geq 0, x\in\X_{j}, V_{\gamma}(x) > \hat{V}_{\gamma}(x)+\!\frac{4 \overline{\Lambda}(x)R}{\gamma N} + \epsilon_{x}\right\}\\
	&\leq P^N\left\{x\in\X_{0},\ V_{\gamma}(x) > \hat{V}_{\gamma}(x) + \frac{2R}{\gamma N} + \epsilon\right\}
	\\
		&\quad + P^N\!\left\{\exists j\geq 1, x\in\X_{j}, V_{\gamma}(x) > \hat{V}_{\gamma}(x) + \frac{2 l_j R}{\gamma N} + \epsilon_{j}\right\}.
\end{align*}
For $\X_j$ as in~\eqref{eq:proofXj}, the bound~\eqref{eq:proofemprad} yields
$$
	\hat{\mathcal{R}}_N(\mathcal{F}(\X_j)) \leq  \frac{l_j R}{N} 
$$
and thus \eqref{eq:proofposterior1} gives, for $j=0$, 
$$
P^N\left\{x\in\X_{0},\ V_{\gamma}(x) > \hat{V}_{\gamma}(x) + \frac{2R}{\gamma N} + \epsilon\right\} \leq 2e^{-2N\epsilon^2/9}
$$
and, for any $j\geq 1$, 
$$
	P^N\left\{ x\in\X_{j},\ V_{\gamma}(x) > \hat{V}_{\gamma}(x) + \frac{2 l_j R}{\gamma N} + \epsilon_{j}\right\}\leq 2e^{-2N\epsilon_j^2/9}.
$$
Since $l_j = 2^j$, we have 
$$
	\epsilon_{j}^2 = \left(\epsilon + \sqrt{\frac{\alpha\log j}{N}}\right)^2 \geq \epsilon^2 + \frac{\alpha\log j}{N} .
$$ 
Thus, with $p=2\alpha/9$, 
\begin{align*}
	e^{-2N\epsilon_j^2/9}&\leq e^{-2N(\epsilon^2 + \alpha\log (j)/N)/9}\\
	& = e^{-2N\epsilon^2/9} e^{-p\log (j)} = \frac{e^{-2N\epsilon^2/9}}{j^p}.
\end{align*}
Applyng the union bound, this leads to 
\begin{align*}
P^N&\left\{\exists j\geq 1, x\in\X_{j},\ V_{\gamma}(x) > \hat{V}_{\gamma}(x) + \frac{4 \overline{\Lambda}(x)R}{\gamma N} + \epsilon_{x}\right\}\\
&\qquad\qquad\leq 2\sum_{j\geq 1} e^{-2N\epsilon_j^2/9} \\
&\qquad\qquad\leq 2 e^{-2N\epsilon^2/9}\sum_{j\geq 1} \frac{1}{j^p} = 2\ \zeta(p) e^{-2N\epsilon^2/9}.
\end{align*}
Therefore, 
\begin{align*}
	P^N\left\{\exists x\in\X,\ V_{\gamma}(x) > \hat{V}_{\gamma}(x) + \frac{4 \overline{\Lambda}(x)R}{\gamma N} + \epsilon_{x}\right\}&\\
	\leq (2+2 \zeta(p))& e^{-2N\epsilon^2/9},
\end{align*}
which, by recalling the value of $\epsilon$ in~\eqref{eq:proofepsilon}, is precisely~\eqref{eq:proofthm3}.

\section{Proof of Theorem~\ref{thm:covering}}
\label{sec:proofcovering}

Throughout this section, we consider a slightly different margin loss function:
\begin{equation}\label{eq:lossinf}
	\ell_{\gamma}(x,\theta) = \I{f(x,\theta) > -\gamma} \geq \I{f(x,\theta) > 0} . 
\end{equation}
We will first prove the following more general result: with probability at least $1-\delta$, for all $x\in\X$, 
\begin{equation}\label{eq:proofcoveringres}
	V(x) \leq V_{\gamma}(x)\leq  \hat{V}_{\gamma}(x) + 2\sqrt{ \hat{V}_{\gamma}(x)\frac{M + \log\frac{4}{\delta}}{N}} +  4\frac{M + \log\frac{4}{\delta}}{N}
\end{equation}
with $V_{\gamma}(x)$ and $\hat{V}_{\gamma}(x)$ computed using~\eqref{eq:lossinf} and 
$$
	M = 144 \sum_{k=1}^C M_k^2 \log(60 M_k N)
$$
with
$$
	M_k = \frac{2^{p_k}  \overline{\varphi}_k\tau_k\Lambda_k \prod_{j=k}^C \overline{\rho}_j}{\gamma}.
$$
Then, Theorem~\ref{thm:covering} will follow by assuming that $f(x,\theta_i)\leq -\gamma$, $i=1,\dots,N$, which ensures that  $\hat{V}_{\gamma}(x) = 0$ in~\eqref{eq:proofcoveringres}. 

This result relies on the following capacity measure. 

\begin{definition}[$\epsilon$-nets and covering numbers]
Given a function class $\F\subset\R^{\Theta}$ and a sequence $\g \theta_N \in \Theta^N$, consider the empirical pseudo-metric defined by
$$
	\forall (f,f^\prime)\in\left(\R^{\Theta}\right)^2,\quad d_{\g \theta_N}(f,f^\prime) = \max_{1\leq i\leq N} |f(\theta_i) - f^\prime(\theta_i)|.
$$
Then, an {\em $\epsilon$-net} of $\F$ is any set $\F'\subset\R^{\Theta}$ such that $\forall f\in\F$, $d_{\g \theta_N}(f,\F') = \min_{f'\in\F'} d_{\g \theta_N}(f,f')< \epsilon$ and the {\em covering number} $\mathcal{N}(\epsilon, \F, \g\theta_N)$ at scale $\epsilon$ of $\F$ is the cardinality of its smallest $\epsilon$-net. 
{\em Uniform covering numbers} are defined by
$$
	\mathcal{N}(\epsilon, \F, N) = \sup_{\g \theta_N\in \Theta^N} \mathcal{N}(\epsilon, \F, \g \theta_N ).
$$
\end{definition}

We will also require the following contraction lemma.
\begin{lemma}\label{lem:contractioncovering}
Let $\varphi: R \to\R$ be a $\overline{\varphi}$-Lipschitz continuous function over its domain $R\subset\R$, i.e., $\forall u,v\in R$, $|\varphi(u)-\varphi(v)|\leq \overline{\varphi}|u-v|$. Then, for all $\epsilon>0$,  
$$
	\N(\epsilon, \varphi\circ\F, N) \leq \N\left(\frac{\epsilon}{\overline{\varphi}}, \F, N\right).
$$
\end{lemma}
\begin{proof}
For any $\g\theta_N=(\theta_i)_{1\leq i\leq N}$, let $\hat{\F}$ be a minimal $\epsilon/\overline{\varphi}$-cover of $\F$ with cardinality $\N(\epsilon/\overline{\varphi},\F, \g\theta_N)$. Then, for any $f\in\F$, there is a $\hat{f}\in\hat{\F}$ such that
\begin{align*}
	d_{\g\theta_N}(\varphi\circ f,\varphi\circ\hat{f}) &= \max_{1\leq i\leq N} |\varphi(f(\theta_i)) - \varphi(\hat{f}(\theta_i))|\\
		&\leq \max_{1\leq i\leq N} \overline{\varphi}|f(\theta_i) - \hat{f}(\theta_i)| \\
		&\leq \overline{\varphi}d_{\g\theta_N}( f,\hat{f})  < \epsilon
\end{align*}
and $\varphi\circ\hat{\F}$ is an $\epsilon$-net of $\varphi\circ\F$ of cardinality $\N(\epsilon/\overline{\varphi},\F, \g\theta_N)$. Taking the supremum over all $\g\theta_N$ concludes the proof.
\end{proof}

The following lemmas will provide bounds on the covering numbers of $\varphi_k\circ\F_k$ and of the combination of two classes.

\begin{lemma}\label{lem:coverpseudolinear}
Given functions $\psi:\Theta\times \R^{n}$, $\phi:\Theta\to\R^n$, $\eta:\Theta\to\R$ and a $\overline{\varphi}$-Lipschitz continuous function $\varphi: \R \to\R$, the uniform covering numbers of the class $\F = \{f_x \in \R^{\Theta}: f_x(\theta) = \varphi\left( \psi(\theta)^\top \phi(x) + \eta(\theta) \right),\ x\in\X \}$ are bounded, for all $\epsilon\in(0,\tau\Lambda]$, by
$$
	\log \N(\epsilon,\mathcal{F}, N) \leq  \frac{36 \overline{\varphi}^2 \tau^2 \Lambda^2}{\epsilon^2} \log \frac{15\overline{\varphi}\tau\Lambda N}{\epsilon}, 
$$
where $\tau = \sup_{\theta\in\Theta} \|\psi(\theta)\|$ and $\Lambda = \sup_{x\in\X}\|\phi(x)\|$. 
\end{lemma}
\begin{proof}
Let $\tilde{\F} =  \{f_x \in \R^{\Theta}: f_x(\theta) =  \psi(\theta)^\top \phi(x) + \eta(\theta) ,\ x\in\X \}$  such that $\F = \varphi\circ \tilde{\F}$. Then, the result follows from Lemma~\ref{lem:contractioncovering} and the following bound on the covering numbers of $\tilde{\F}$.

For all $\epsilon\in(0,\tau\Lambda]$, Theorem~4 in \cite{Zhang02} yields an $\epsilon$-net of $\F' = \{f'_x \in \R^{\Theta}: f'_x(\theta) = \psi(\theta)^\top \phi(x),\ x\in\X \}$ with cardinality $n_1$ such that
\begin{align*}
	\log n_1 &\leq \frac{36\tau^2 \Lambda^2}{\epsilon^2} \log \left(2N\left\lceil \frac{4\tau\Lambda}{\epsilon} +2\right\rceil + 1\right) \\
		&< \frac{36\tau^2 \Lambda^2}{\epsilon^2} \log \frac{15\tau\Lambda N}{\epsilon}.
\end{align*}
For all $\hat{f}'_x$ in this $\epsilon$-net of $\F'$, let $\hat{f}_x = \hat{f}'_x + \eta$ and note that any function $f_x\in\tilde{\F}$ can be decomposed as $f_x = f'_x + \eta$. Thus, for any $f_x\in\tilde{\F}$, there is an $\hat{f}_x$ such that 
\begin{align*}
	d(f_x, \hat{f}_x) &= \max_{1\leq i\leq N} |f'_x(\theta_i) + \eta(\theta_i) - \hat{f}'_x(\theta_i) - \eta(\theta_i)|\\
		& = \max_{1\leq i\leq N} |f'_x(\theta_i) - \hat{f}'_x(\theta_i)|< \epsilon
\end{align*}
and the set of functions $\hat{f}_x$ thus built forms an $\epsilon$-net of $\tilde{\F}$. Since the cardinality of this set is $n_1$, we conclude that $\N(\epsilon,\tilde{\F}, N)\leq n_1$.
\end{proof}

\begin{lemma}\label{lem:binarycovering}
Let $\F_1$ and $\F_2$ be subsets of $\R^{\Theta}$ and $\G(\F_1,\F_2)$ be either $\F_1+\F_2$, $\F_1-\F_2$, $\max(\F_1,\F_2)$ or $\min(\F_1,\F_2)$. 
Then, for any $\epsilon>0$,
$$
	\N(\epsilon,\G(\F_1,\F_2), N) \leq \N\left(\frac{\epsilon}{2^a},\F_1, N\right) \N\left(\frac{\epsilon}{2^a},\F_2, N\right) ,
$$
where $a$ is $1$ if $\G= \F_1+\F_2$ or $\G= \F_1-\F_2$ and $0$ otherwise.
\end{lemma}
\begin{proof}
Let $\hat{\F}_1$ and $\hat{\F}_2$ denote minimal $\epsilon$-nets of $\F_1$ and $\F_2$ of cardinality $\N(\epsilon,\F_1, N)$ and $\N(\epsilon,\F_2, N)$, respectively. Then, for any $f_1\in\F_1$ and $f_2\in\F_2$, there are $\hat{f}_1\in\hat{\F}_1$ and $\hat{f}_2\in\hat{\F}_2$ such that
\begin{align*}
	d_{\g\theta_N}(f_1+f_2,&\ \hat{f}_1+\hat{f}_2) \\
	&= \max_{1\leq i\leq N} |f_1(\theta_i) + f_2(\theta_i) -\hat{f}_1(\theta_i) - \hat{f}_2(\theta_i)| \\
	& \leq \max_{1\leq i\leq N} |f_1(\theta_i)  -\hat{f}_1(\theta_i)| + | f_2(\theta_i) - \hat{f}_2(\theta_i)| \\
	&\leq d_{\g\theta_N}(f_1, \hat{f}_1) + d_{\g\theta_N}( f_2, \hat{f}_2) < 2\epsilon.
\end{align*}
Thus, $\{\hat{f}_1+\hat{f}_2,\ \hat{f}_1\in\hat{\F}_1, \hat{f}_2\in\hat{\F}_2\}$ is a $(2\epsilon)$-net of $\F_1+\F_2$ of cardinality no larger than $|\hat{\F}_1|\cdot|\hat{\F}_2|$, and the result follows by rescaling $\epsilon$ and setting $a=1$. 

The same argument yields $d_{\g\theta_N}(f_1-f_2, \hat{f}_1-\hat{f}_2) < 2\epsilon$, and the result for $\F_1-\F_2$ with again $a=1$. The proof for $\max(\F_1,\F_2)$ and $\min(\F_1,\F_2)$ can be found in Lemma~13 of \cite{Lauer20}, where the scale $\epsilon$ is not affected by the operation and $a$ can be set to $0$.
\end{proof}

We are now ready to formulate the bound on the covering numbers of the class of functions $f$ as in~\eqref{eq:fminmax}, which will play a role similar to Theorem~\ref{thm:radboolean} in the proof of~\eqref{eq:proofcoveringres}.
\begin{lemma}\label{lem:proofcoveringmain}
Let $\F$ be a function class as in~\eqref{eq:F} with $f$ of the form~\eqref{eq:fminmax}. Then,
\begin{align*}
	\log \N(\epsilon,\mathcal{F}, N) \leq  \frac{36}{\epsilon^2} \sum_{k=1}^C & 4^{p_k}  \left(\prod_{j=k}^C \overline{\rho}_j\right)^2\overline{\varphi}_k^2 \tau_k^2 \Lambda_k^2 \\
	& \times \log \frac{15\cdot 2^{p_k} (\prod_{j=k}^C \overline{\rho}_j) \overline{\varphi}_k\tau_k\Lambda_k N}{\epsilon}, 
\end{align*}
where $\tau_k = \sup_{\theta\in\Theta} \|\psi_k(\theta)\|$, $\Lambda_k =\sup_{x\in\X} \|\phi_k(x)\|$, $\overline{\rho}_1=1$, $\overline{\rho}_k$ for $k\geq 2$ and $\overline{\varphi}_k$ for $k\geq 1$ are the Lipschitz constants of the functions $\rho_k$ and $\varphi_k$, respectively, and $p_k=\sum_{j=k}^C a_j$ with $a_j = 1$ if the binary operation $g_j$ is a sum or a difference and $0$ otherwise. 
\end{lemma}
\begin{proof}
Consider the notations defined in~\eqref{eq:recursiveF}. For $C=1$, $\F= \varphi_1\circ \F_1$ and Lemma~\ref{lem:coverpseudolinear} yields the desired result (note that $p_1$ is always zero in this case). 

Assume now that the statement holds for $\F^C$. Then, Lemma~\ref{lem:contractioncovering} yields
\begin{align*}
	\log &\ \N(\epsilon,\mathcal{F}^{C+1}, N) \\&\leq \log \N\left(\frac{\epsilon}{\overline{\rho}_{C+1}},\mathcal{G}_{C+1}(\F^C, \varphi_{C+1}\circ\F_{C+1}), N\right). 
\end{align*}
Let $\epsilon' = \epsilon/\overline{\rho}_{C+1}$. 
Then, Lemma~\ref{lem:binarycovering} gives
\begin{align*}
	\log &\ \N(\epsilon',\mathcal{G}_{C+1}(\F^C, \varphi_{C+1}\circ\F_{C+1}), N) \\
	&\leq \log \N\left(\frac{\epsilon'}{2^{a_{C+1}}},\F^C, N\right) \\
	&\quad + \log \N\left(\frac{\epsilon'}{2^{a_{C+1}}}, \varphi_{C+1}\circ\F_{C+1}, N\right),
\end{align*}
with $a_{C+1}=0$ if $g_{C+1}$ is the maximum or minimum operator and $a_{C+1}=1$ if $g_k$ is a sum or a difference. Using this with the assumption on the covering numbers of $\F^{C}$ and Lemma~\ref{lem:coverpseudolinear} further leads to
\begin{align*}
	\log\ &\N(\epsilon',\mathcal{G}_{C+1}(\F^C, \varphi_{C+1}\circ\F_{C+1}), N) 
	\\&\leq \frac{36 \cdot 4^{a_{C+1}}}{\epsilon'^2} \sum_{k=1}^C 4^{p_k} \left(\prod_{j=k}^C \overline{\rho}_j\right)^2\overline{\varphi}_k^2 \tau_k^2 \Lambda_k^2 \\
			&\qquad\qquad\qquad\ \times \log \frac{15\cdot 2^{a_{C+1}}\cdot 2^{p_k} (\prod_{j=k}^C \overline{\rho}_j) \overline{\varphi}_k\tau_k\Lambda_k N}{\epsilon'}\\
	&\quad 	+ \frac{36 \cdot 4^{a_{C+1}}  \overline{\rho}_{C+1}^2\overline{\varphi}_{C+1}^2 \tau_{C+1}^2 \Lambda_{C+1}^2}{\epsilon'^2} \\
	&\qquad\quad \times \log \frac{15\cdot 2^{a_{C+1}}  \overline{\rho}_{C+1} \overline{\varphi}_{C+1}\tau_{C+1}\Lambda_{C+1} N}{\epsilon'}.
\end{align*}
By replacing $\epsilon'$ by $\epsilon/\overline{\rho_{C+1}}$ and letting $p_k$ take its new value $p_k+a_{C+1}$, we conclude that the statement holds for $f^{C+1}$.
\end{proof}

Now, the proof of~\eqref{eq:proofcoveringres} (and thus Theorem~\ref{thm:covering}) stems from the relative deviation bounds of \cite{Cortes21}. 
In particular, by Corollary~5 of \cite{Cortes21},  
with probability at least $1-\delta$, 
\begin{align*}
	V(x) \leq V_{\gamma}(x) & \leq  \hat{V}_{\gamma}(x) \\
	&\quad + 2\sqrt{ \hat{V}_{\gamma}(x)\frac{\log  \mathcal{N}( \frac{\gamma}{2},\F_{\gamma}, 2N) + \log\frac{4}{\delta}}{N}} \\
	&\quad +  4\frac{\log \mathcal{N}( \frac{\gamma}{2}, \F_{\gamma},2N) + \log\frac{4}{\delta}}{N},
\end{align*}
where $\F_{\gamma}$ is the class of functions from $\F$ clipped at $\gamma$: 
\begin{align*}
	\F_{\gamma} = \{f_{x,\gamma} &\ \in [-\gamma,\gamma]^{\Theta} :\\
	& f_{x,\gamma}(\theta) = \max\{ -\gamma, \min\{ \gamma, f_x(\theta)\}\},\ f_x\in\F\}.
\end{align*}
Since clipping $\F$ amounts to composing its functions with the $1$-Lipschitz function $\max\{ -\gamma, \min\{ \gamma, \cdot\}\}$, Lemma~\ref{lem:contractioncovering} ensures that $\mathcal{N}(\frac{\gamma}{2},\F_{\gamma},  2N)\leq \mathcal{N}(\frac{\gamma}{2}, \F, 2N)$. Then, Lemma~\ref{lem:proofcoveringmain} applied with $\epsilon=\gamma/2$ gives the result~\eqref{eq:proofcoveringres}.

\end{document}